\documentclass[11pt]{article}
\usepackage{amsthm, amsmath, amssymb, amsfonts, url, booktabs, tikz, setspace, fancyhdr, bm}
\usepackage{hyperref}
\usepackage{amsthm}
\usepackage{cancel}
\usepackage{geometry}
\usepackage{hyperref, enumerate}
\usepackage[shortlabels]{enumitem}
\usepackage[babel]{microtype}
\usepackage[english]{babel}
\usepackage[capitalise]{cleveref}
\usepackage{comment}
\usepackage{bbm}
\usepackage{csquotes}
\usepackage{mathabx}
\usepackage{tikz}
\usetikzlibrary{shapes.geometric, positioning}
\usepackage{subcaption}
\usepackage{graphicx}
\usepackage{float}
\usepackage{xcolor}
\usepackage{tikz-cd}
\usepackage{adjustbox}
\usepackage{caption}
\usepackage{mathrsfs}
\usepackage{pgfplots}
\pgfplotsset{compat=newest}
\usepackage{subcaption}

\usepgfplotslibrary{fillbetween}

\usetikzlibrary{decorations.pathmorphing}
\usetikzlibrary{positioning, arrows.meta, shapes.geometric}

\counterwithin{figure}{section}

\newtheorem{theorem}{Theorem}[section]
\newtheorem{prop}[theorem]{Proposition}
\newtheorem{lemma}[theorem]{Lemma}
\newtheorem{cor}[theorem]{Corollary}
\newtheorem{conj}[theorem]{Conjecture}
\newtheorem{claim}[theorem]{Claim}

\newtheorem{fact}[theorem]{Fact}
\newtheorem{obs}[theorem]{Observation}

\theoremstyle{definition}
\newtheorem{definition}[theorem]{Definition}
\newtheorem*{defn-non}{Definition}

\newtheorem{problem}[theorem]{Problem}
\definecolor{rosepink}{RGB}{255,102,204}
\definecolor{dateplum}{HTML}{993366}
\definecolor{darkdateplum}{RGB}{128,0,32}
\definecolor{lightdateplum}{RGB}{219,112,147}
\definecolor{darkred}{RGB}{139,0,0}
\definecolor{lightred}{RGB}{240,130,100}

\definecolor{curve1}{RGB}{228,26,28}    
\definecolor{curve2}{RGB}{55,126,184}   
\definecolor{curve3}{RGB}{77,175,74}    
\definecolor{curve4}{RGB}{152,78,163}   
\definecolor{curve5}{RGB}{255,127,0}    
\definecolor{limitcurve}{RGB}{0,0,0}    

\newtheorem{rmk}[theorem]{Remark}

\newlist{Case}{enumerate}{3}
\setlist[Case, 1]{%
    label           =   {\bfseries Case \arabic*.},
    labelindent=1em ,labelwidth=1cm, labelsep*=1em, leftmargin =!
}
\setlist[Case, 2]{%
    label           =   {\bfseries Subcase \arabic{Casei}.\arabic*.},
    labelindent=-1em ,labelwidth=1cm, labelsep*=1em, leftmargin =!
}
\setlist[Case, 3]{%
    label           =   {\bfseries Subsubcase \arabic{Casei}.\arabic{Caseii}.\arabic*.},
    labelindent=-1em ,labelwidth=1cm, labelsep*=1em, leftmargin =!
}

\newenvironment{poc}{\begin{proof}[Proof of claim]}{\end{proof}}

\newcommand{\C}[1]{{\protect\mathcal{#1}}}
\newcommand{\B}[1]{{\bf #1}}
\newcommand{\I}[1]{{\mathbbm #1}}

\usepackage{todonotes} 

\newcommand{\eps}{\varepsilon}

\newcommand{\VC}{\mathrm{VC}}

\title{Interpolating chromatic and homomorphism thresholds}
\author{
Xinqi Huang\thanks{School of Mathematical Sciences, University of Science and Technology of China, Hefei, Anhui 230026,
China.
Emails:
\{huangxq, rong\_ming\_yuan\}@mail.ustc.edu.cn. Xinqi Huang was supported by the National Key Research and
Development Programs of China 2023YFA1010200 and 2020YFA0713100, the NSFC under Grants No. 12171452 and No. 12231014, and the Innovation Program for Quantum Science and Technology 2021ZD0302902. Mingyuan Rong was supported by National Key Research and Development Program of China 2023YFA1010201
and the NSFC under Grants No. 12125106.
}
\and
Hong Liu\thanks{Extremal Combinatorics and Probability Group (ECOPRO), Institute for Basic Science (IBS), Daejeon, South Korea. Emails: {\texttt \{hongliu, zixiangxu\}@ibs.re.kr}. Supported by the Institute for Basic Science (IBS-R029-C4).}
\and
Mingyuan Rong\footnotemark[1]
\and
Zixiang Xu\footnotemark[2]
}

\begin{document}
\date{}
\maketitle
\begin{abstract}
The problem of chromatic thresholds seeks for minimum degree conditions that ensure dense $H$-free graphs to have a bounded chromatic number, or equivalently a bounded size homomorphic image. The strengthened homomorphism thresholds problem further requires that the homomorphic image itself is $H$-free. The purpose of this paper is two-fold. 

First, we define a generalized notion of threshold which encapsulates and interpolates chromatic and homomorphism thresholds via the theory of VC-dimension. Our first main result shows a smooth transition between these two thresholds when varying the restrictions on homomorphic images. 
In particular, we proved that for \(t \geq s \geq 3\) and $\varepsilon>0$, if \(G\) is an $n$-vertex \(K_s\)-free graph with VC-dimension $d$ and $\delta(G) \geq \left(\frac{(s-3)(t-s+2)+1}{(s-2)(t-s+2)+1} + \varepsilon\right)n\), then \(G\) is homomorphic to a \(K_t\)-free graph \(H\) with \(|H| = O_{s,t,d,\varepsilon}(1)\). Moreover, we construct graphs showing that this minimum degree condition is optimal. This extends and unifies the results of Thomassen, {\L}uczak and Thomass\'e, and Goddard, Lyle and Nikiforov, and provides a deeper insight into the cause of existences of homomorphic images with various properties.

Second, we introduce the blowup threshold $\delta_{\textup{B}}(H)$ as the infimum $\alpha$ such that every $n$-vertex maximal $H$-free graph $G$ with $\delta(G)\ge\alpha n$ is a blowup of some $F$ with $|F|=O_{\alpha,H}(1)$. This notion strengthens homomorphism thresholds. While the homomorphism thresholds for odd cycles remain unknown, we prove that $\delta_{\textup{B}}(C_{2k-1})=\frac{1}{2k-1}$ for any integer $k\ge 2$. This strengthens the result of Ebsen and Schacht and answers a question of Schacht and shows that, in sharp contrast to the chromatic thresholds, 0 is an accumulation point for blowup thresholds.

Our proofs mix tools from VC-dimension theory and an iterative refining process to find the desired homomorphic images, and draw connection to a problem concerning codes on graphs. 
\end{abstract}

\section{Introduction}

\subsection{Overview}
A central problem in extremal graph theory is to understand the properties of graphs that exclude specific subgraphs. These properties include subgraphs count, chromatic number, and more intricate structural characteristics. Regarding  chromatic number, as one of the classical applications of probabilistic methods, Erd\H{o}s~\cite{1959ErdosRandom} proved the existence of graphs with arbitrarily large girth and chromatic number, showing that forbidden subgraph conditions alone is not sufficient to bound chromatic number. Regarding edge count, the cornerstone theorem of Erd\H{o}s, Stone and Simonovits~\cite{1966ES,1946ErodsBAMS} asserts that the maximum number of edges in an $n$-vertex $H$-free graph is $(1-\frac{1}{\chi(H)-1}+o(1))\binom{n}{2}$, where $\chi(H)$ is the chromatic number of $H$. Back to chromatic number, Andr\'{a}sfai, Erd\H{o}s, and S\'{o}s~\cite{1974ErdosSos} proved that every $n$-vertex $K_{s}$-free graph with minimum degree larger than $\frac{3s-7}{3s-4} n$ has chromatic number at most $s-1$. This leads to the concept of
\emph{chromatic threshold}, first introduced by Erd\H{o}s and Simonovits~\cite{1973ErdosSimonovits} in 1973, which studies the optimal minimum degree  that guarantees $H$-free graphs to have bounded chromatic number. Formally, the chromatic threshold of $H$ is defined as
$$\delta_{\chi}(H):=\inf\big\{\alpha\ge 0:\exists~C=C(\alpha,H)\  \textup{s.t.~} \forall~n\textup{-vertex\ }H\textup{-free}~G,~\delta(G)\ge \alpha n\Rightarrow \chi(G)\le C \big\}.
$$
The first non-trivial case arises when $H$ is a triangle. A construction of Hajnal utilizing Kneser graphs shows that $\delta_{\chi}(K_3)\ge \frac{1}{3}$~\cite{1973ErdosSimonovits}. A matching upper bound was proved by Thomassen~\cite{2002Thomassen}. Later, Goddard and Lyle~\cite{2011JGTKrChromatic} and independently Nikiforov~\cite{2010arxivKrfree} solved the problem for cliques: $\delta_{\chi}(K_s)= \frac{2s-5}{2s-3}$. For odd cycles $C_{2k-1}$, Thomassen~\cite{2007OddCycleChromatic} proved that $\delta_{\chi}(C_{2k-1})=0$ for any $k\ge 3$. This topic has attracted significant attention~\cite{2011Unpubilished,1982Haggkvist,1995DMJin, 2010ColoringViaVCDim}. Finally, Allen, B\"{o}ttcher, Griffiths, Kohayakawa, and Morris~\cite{2013advAllChromatic} determined the chromatic thresholds for all graphs $H$.

Given graphs $G$ and $F$, we say that $G$ is \emph{homomorphic} to $F$, denoted by $G\xrightarrow{\textup{hom}} F$, if there exists a homomorphism $\varphi:V(G)\rightarrow V(F)$ that preserves adjacencies, that is, if $uv\in E(G)$, then $\varphi(u)\varphi(v)\in E(F)$. Note that having chromatic number $k$ is equivalent to being homomorphic to a clique $K_{k}$. Thomassen~\cite{2002Thomassen} proposed the problem of what minimum degree guarantees an $H$-free graph to have a bounded homomorphic image that itself is also $H$-free. Formally, the \emph{homomorphism threshold} of $H$ is defined as
$$
    \delta_{\textup{hom}}(H):=\inf\big\{\alpha\ge 0:\exists~H\textup{-free}~F=F(\alpha,H)\  \textup{s.t.~}\forall~n\textup{-vertex\ }H\textup{-free}~G,~\delta(G)\ge \alpha n\Rightarrow G\xrightarrow{\textup{hom}} F  \big\}.
$$
By definition, $\delta_{\textup{hom}}(H)\ge\delta_{\chi}(H)$. {\L}uczak~\cite{2006CombTriangleHom} first solved the triangle case, and for any $s\ge 3$, Goddard and Lyle~\cite{2011JGTKrChromatic} proved that $\delta_{\textup{hom}}(K_{s})=\delta_{\chi}(K_s)=\frac{2s-5}{2s-3}$. The proofs of {\L}uczak~\cite{2006CombTriangleHom}, and Goddard and Lyle~\cite{2011JGTKrChromatic} employed Szemer\'{e}di's regularity lemma~\cite{1978OriginalRegularity}, and therefore the size of the $K_s$-free homomorphic image $F$ in their result is extremely large. Via a clever probabilistic argument, Oberkampf and Schacht~\cite{2020CPCProb} significantly improved the size of the homomorphic image. The best known bound to date is due to Liu, Shangguan, Skokan and Xu~\cite{2024GraphToGeom}. Compared to chromatic threshold, not much is known besides cliques for homomorphism threshold. For odd cycles, Letzter and Snyder~\cite{2019JGTC3C5} showed that $\delta_{\textup{hom}}(C_{5})\le\frac{1}{5}$, and later Obsen and Schacht~\cite{2020COMBHomoOddCycle} extended it to $\delta_{\textup{hom}}(C_{2k-1})\le\frac{1}{2k-1}$ for any $k\ge 3$. Very recently, using tools from topology, Sankar~\cite{2022Maya} provided the first instance of separation between chromatic and homomorphism threshold: $\delta_{\textup{hom}}(C_{2k-1})>k^{-(8+o(1))k}>0=\delta_{\chi}(C_{2k-1})$ for any $k\ge 3$. In general, studying the structural properties of dense graphs that do not contain small odd cycles is a topic of great interest. We refer the readers to~\cite{2003Bollobas,2024SIDMAJanzer,2023Skokan3colorable,1998JGT,2024Hou,1982Haggkvist,2024Wang,2004JGTNikiforov,2024PengC2k+1,2024JGTPeng} and the references therein.

\subsection{Interpolating chromatic and homomorphism threshold via VC-dimension}

The starting point of our work is to understand why cliques have the same chromatic and homomorphism thresholds $\delta_{\chi}(K_s)=\delta_{\textup{hom}}(K_{s})$, even though homomorphism thresholds have extra restrictions on homomorphic images. As we shall see, this `coincidence' is closely related to the \emph{Vapnik-Chervonenkis dimension} (\emph{VC-dimension} for short). The VC-dimension is a key concept in statistical learning theory, which reflects the complexity of patterns a model can capture and holds significance across various fields including machine learning and computational complexity. Recently, an increasing number of works~\cite{2024BVCSunflower,2025UniformBVC,2019DCGBVCEH,2021ErdosSchur,2023BVCSunflower,2024FranklPach,2007JACMZ,2023BVCErdosHajnal,2023SukMatchingLemma} have recognized the application of VC-dimension in extremal combinatorics. Given a set system $\mathcal{F} \subseteq 2^V$ on the ground set $V$, its VC-dimension is the largest size of subset $S\subseteq V$ such that for every $S'\subseteq S$, there exists a member $F_{S'}\in\mathcal{F}$ such that $F_{S'}\cap S=S'$. The VC-dimension of a graph $G$, denoted by $\textup{VC}(G)$, is defined to be the VC-dimension of the set system $\{N_{G}(v):v\in V(G)\}$.

In an unpublished work, {\L}uczak and Thomass{\'e}~\cite{2010ColoringViaVCDim} discovered an interesting phenomenon that VC-dimension heavily affects the chromatic number of a dense triangle-free graph. More precisely, they showed that an $n$-vertex triangle-free graph $G$ with bounded VC-dimension and $\delta(G)=\Omega(n)$ must have bounded chromatic number. This shows that the bottleneck for $\delta_{\chi}(K_3)$ getting stuck at $\frac{1}{3}$ is the VC-dimension. To look deeper into the influence of VC-dimension for thresholds, Liu, Shangguan, Skokan and Xu~\cite{2024GraphToGeom} introduced the following natural variant with bounded VC-dimension:
\begin{align*}
    \delta_{\chi}^{\textup{VC}}(H) :=  \inf\big\{\alpha\ge 0: &~\forall d\in\mathbb{N}, \exists~C=C(\alpha,H,d)\ 
     \textup{s.t.~} \forall~n\textup{-vertex\ }H\textup{-free}~G,\\&~\VC(G)\le d,~\delta(G)\ge \alpha n\Rightarrow \chi(G)\le C \big\},
\end{align*}
and
\begin{align*}
    \delta_{\textup{hom}}^{\textup{VC}}(H) :=  \inf\big\{\alpha\ge 0: &~\forall d\in\mathbb{N}, \exists~H\textup{-free}~\textup{graph}~F=F(\alpha,H,d)\ 
     \textup{s.t.~} \forall~n\textup{-vertex\ }H\textup{-free}~G,\\&~\VC(G)\le d,~\delta(G)\ge \alpha n\Rightarrow G\xrightarrow{\textup{hom}} F  \big\}.
\end{align*}
The above result of {\L}uczak and Thomass{\'e}~\cite{2010ColoringViaVCDim} implies that $\delta_{\chi}^{\textup{VC}}(K_3)=0$. This was recently generalized by Liu, Shangguan and Xue~\cite{2025StabChromaticThresholds} who showed that $\delta_{\chi}^{\textup{VC}}(K_s)=\frac{s-3}{s-2}<\frac{2s-5}{2s-3}=\delta_{\chi}(K_s)$.

In contrast, Liu, Shangguan, Skokan and Xu~\cite{2024GraphToGeom} observed that the homomorphic thresholds of cliques do not decrease under the additional assumption of bounded VC-dimension (see~\cref{rmk:Kr}):
\begin{equation}\label{eq:VC-effect}
   \delta_{\textup{hom}}^{\textup{VC}}(K_s)=\delta_{\textup{hom}}(K_s)=\delta_{\chi}(K_s)=\frac{2s-5}{2s-3}>\frac{s-3}{s-2}=\delta_{\chi}^{\textup{VC}}(K_s).
\end{equation}

Recall that the difference between $\delta_{\chi}$ and $\delta_{\textup{hom}}$ is the restriction on homomorphic images. We introduce the following version of thresholds, allowing varying restrictions on homomorphic images. This generalization encapsulates and unifies all of the above variants. Given families of graphs $\mathcal{G}_{1},\mathcal{G}_{2}$, we define the \emph {bounded-VC homomorphism thresholds} of $(\mathcal{G}_{1},\mathcal{G}_{2})$ as
\begin{align*}
    \delta_{\textup{hom}}^{\textup{VC}}(\mathcal{G}_{1};\mathcal{G}_{2}) :=  \inf\big\{\alpha\ge 0: &~\forall d\in\I N, \exists~\mathcal{G}_{2}\textup{-free}~\textup{graph}~F=F(\alpha,\mathcal{G}_{1},\mathcal{G}_{2},d)\ 
     \textup{s.t.~} \forall~n\textup{-vertex\ }\mathcal{G}_{1}\textup{-free}~G,\\&~\VC(G)\le d,~\delta(G)\ge \alpha n\Rightarrow G\xrightarrow{\textup{hom}} F  \big\}.
\end{align*}

Our first result resolves this generalized problem for all pairs of cliques.

\begin{theorem}\label{thm:KsKt}
    For any positive integers $t\ge s\ge 3$, we have 
    $$\delta_{\textup{hom}}^{\textup{VC}}(K_{s};K_{t})=\frac{(s-3)(t-s+2)+1}{(s-2)(t-s+2)+1}.$$
\end{theorem}

While VC-dimension does not appear to directly affect $\delta_{\textup{hom}}(K_{s})$, \cref{thm:KsKt} provides a finer picture on~\eqref{eq:VC-effect} revealing how it enters the game. In particular, \cref{thm:KsKt} gives a smooth interpolation between $\delta_{\textup{hom}}^{\textup{VC}}(K_s)$ and $\delta_{\chi}^{\textup{VC}}(K_s)$  as follows (see~\cref{fig:3d-curves} for a depiction):
$$
\delta_{\textup{hom}}^{\textup{VC}}(K_s)=\delta_{\textup{hom}}^{\textup{VC}}(K_s; K_s)=\frac{2s-5}{2s-3}>\delta_{\textup{hom}}^{\textup{VC}}(K_s; K_{s+1})>\cdots >\frac{s-3}{s-2}=\lim_{t\to\infty}\delta_{\textup{hom}}^{\textup{VC}}(K_s; K_t)=\delta_{\chi}^{\textup{VC}}(K_s).
$$

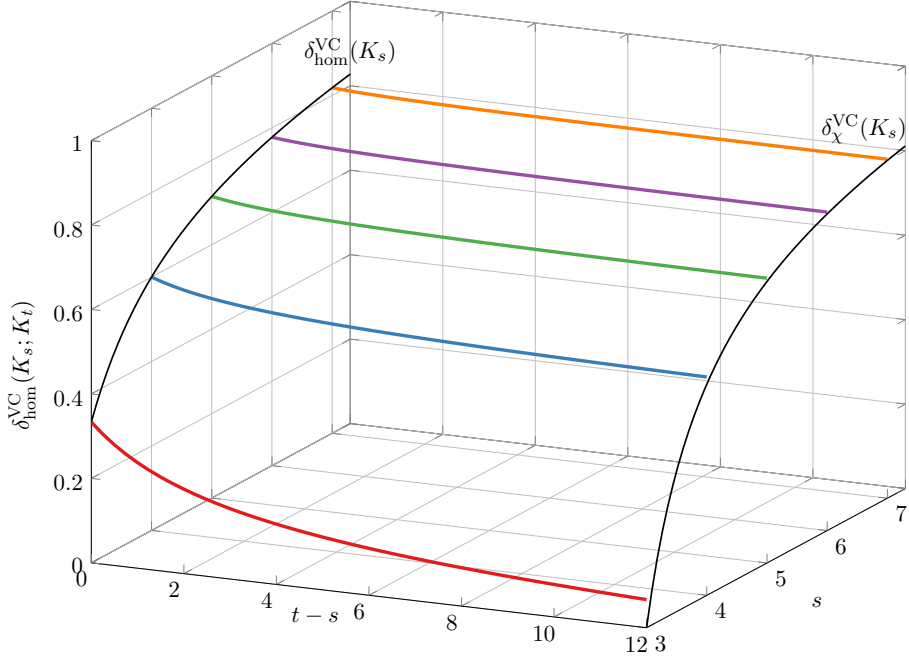
\begin{figure}[ht]
\centering
\begin{tikzpicture}[scale=0.8] 
\begin{axis}[
    width=0.9\textwidth,
    height=12cm,
    view={25}{20},
    xlabel={$t-s$}, xlabel style={anchor=south east},
    ylabel={$s$}, ylabel style={anchor=north west},
    zlabel={$\delta_{\textup{hom}}^{\textup{VC}}(K_s;K_t)$},
    xmin=0, xmax=12,
    ymin=3, ymax=7.3,  
    zmin=0, zmax=1,
    grid=major,
    legend style={at={(0.5,1.05)},anchor=south},
]

\addplot3[
    color=curve1,
    ultra thick,
    samples=50,
    samples y=0,
    domain=3:15,
] (x-3,3,{1/x});

\addplot3[
    color=curve2,
    ultra thick,
    samples=50,
    samples y=0,
    domain=4:16,
] (x-4,4,{(x-1)/(2*x-3)});

\addplot3[
    color=curve3,
    ultra thick,
    samples=50,
    samples y=0,
    domain=5:17,
] (x-5,5,{(2*x -5)/(3*x -8)});

\addplot3[
    color=curve4,
    ultra thick,
    samples=50,
    samples y=0,
    domain=6:18,
] (x-6,6,{(3*x -11)/(4*x -15)});

\addplot3[
    color=curve5,
    ultra thick,
    samples=50,
    samples y=0,
    domain=7:19,
] (x-7,7,{(4*x -19)/(5*x -24)});

\addplot3[
    color=limitcurve,
    thick,
    samples=50,
    samples y=0,
    domain=3:7.3,
] (0,x,{(2*x -5)/(2*x -3)}); 

\addplot3[
    color=limitcurve,
    thick,
    samples=50,
    samples y=0,
    domain=3:7.3,
] (12,x,{(x -3)/(x -2)}); 

\node[anchor=south] at (axis cs:0,7.3,0.83) {$\delta_{\textup{hom}}^{\textup{VC}}(K_s)$};

\node[anchor=south east] at (axis cs:12.2,7.3,0.8) {$\delta_{\chi}^{\textup{VC}}(K_s)$};

\end{axis}
\end{tikzpicture}
\caption{The $x,y,z$-axes correspond to the values of $t-s$, $s$ and $\delta_{\textup{hom}}^{\textup{VC}}(K_s;K_t)$, respectively. The curve on the plane $x=0$ corresponds to $\delta_{\textup{hom}}^{\textup{VC}}(K_s)$; the curves on the planes $y=s$ correspond to $\delta_{\textup{hom}}^{\textup{VC}}(K_s;K_t)$ respectively and their limits when $t\to \infty$ form the curve $\delta_{\chi}^{\textup{VC}}(K_s)$ on the plane $x=+\infty$.}
\label{fig:3d-curves}
\end{figure}

Let $\mathscr{C}_{2k+1}:=\{C_{3},C_{5},\ldots,C_{2k+1}\}$. Notice that the constructions in~\cite{2022Maya} and~\cite{2020COMBHomoOddCycle} both have bounded VC-dimension (see~\cref{prop:Akr}), which implies that $\delta_{\textup{hom}}^{\textup{VC}}(C_{2k+1}) > 0$ and $\delta_{\textup{hom}}^{\textup{VC}}(\mathscr{C}_{2k+1}) = \frac{1}{2k+1}$.
 Our next result shows that there is a similar transition from $\delta_{\textup{hom}}^{\textup{VC}}(\mathscr{C}_{2k+1})=\frac{1}{2k+1}$ to $\delta_{\chi}^{\textup{VC}}(\mathscr{C}_{2k+1}) = 0$.

\begin{theorem}\label{thm:ByproductOne}
For any integer $k\ge 2$, $\delta_{\textup{hom}}^{\textup{VC}}(\mathscr{C}_{2k+1};\mathscr{C}_{2k-1})=0$ and $\delta_{\textup{hom}}^{\textup{VC}}(\{C_{2k+1},C_{2k-1}\};C_{2k-1})=0.$
\end{theorem}

Notably, we develop a general framework (see~\cref{lemma:main}) that not only establishes~\cref{thm:ByproductOne}, but also extends previous results on dense graphs without small odd cycles (see~\cref{thm:SharpThreshold} and~\cref{thm:ByproductThree}). These results will be introduced separately in the following sections.

\subsection{Blowup thresholds}
In the second part of our work, we introduce blowup thresholds, a notion that we believe is more natural than homomorphism thresholds. For a graph $F$, we write $F[t]$ for the \emph{$t$-blowup} of $F$, which is obtained by replacing each vertex of $F$ with an independent set of size $t$, and each edge with a complete bipartite graph $K_{t,t}$. We simply write $F[\cdot]$ when referring to a blowup of $F$ without specifying  blowup size. Clearly, if $G=F[\cdot]$, then $G$ is homomorphic to $F$ and if $G$ is $H$-free, so is $F$. In previous work on homomorphism thresholds for cliques~\cite{2024GraphToGeom,2020CPCProb}, it is in fact demonstrated that under the same min-degree condition, if $G$ is maximal $K_s$-free, then it is a blowup of a bounded size graph. However, for other forbidden graphs, such as odd cycles, the situation is far less clear, which leads us to wonder what min-degree condition forces a maximal $H$-free graph to be a blowup of bounded size graph. Formally, we define the \emph{blowup threshold} of $H$ as
\begin{equation*}
\delta_{\textup{B}}(H):=\inf\big\{\alpha\ge 0:\exists~F(\alpha,H)\  \textup{s.t.~} \forall~n\textup{-vertex\ maximal\ }H\textup{-free}~G,~\delta(G)\ge \alpha n\Rightarrow G=F[\cdot]\big\}.
\end{equation*}
By definitions, we see that $\delta_{\textup{B}}(H)\ge\delta_{\textup{hom}}(H)\ge\delta_{\chi}(H)$ and equalities hold for cliques, i.e.,~for any $s\ge 3$, $\delta_{\textup{B}}(K_{s})=\delta_{\textup{hom}}(K_{s})=\delta_{\chi}(K_{s})$. Interestingly, while the homomorphism thresholds for odd cycles are difficult to determine, we are able to resolve the blowup thresholds counterparts.
\begin{theorem}\label{thm:SharpThreshold}
   Let $\varepsilon>0$, $k\ge 3$ and $G$ be an $n$-vertex maximal $C_{2k-1}$-free graph. If $\delta(G)\ge (\frac{1}{2k-1}+\varepsilon)n$, then $G=H[\cdot]$ for some $H$, where $|H|\le \textup{tw}_{k}(r)$, $r\le e(d+1)(\frac{12ke}{\varepsilon})^{d}$ and $d\le (2k-1)^3\binom{2k-2}{k-1} 2^{(2k-1)^{2k-2}}$. 

   On the other hand, there exists $n$-vertex maximal $C_{2k-1}$-free graph with minimum degree at least $\frac{n}{2k-1}$ which is not a blowup of any smaller graph.
   
   In other words, 
   $$\delta_{\textup{B}}(C_{2k-1})=\frac{1}{2k-1}.$$
\end{theorem}

The first part of~\cref{thm:SharpThreshold} 
improves the result of Ebsen and Schacht~\cite{2020COMBHomoOddCycle} and answers a question of Schacht~\cite{SchachtPC}. Compared to $\delta_{\chi}(C_{2k-1})=0$, we see a clear difference between $\delta_{\textup{B}}$ and $\delta_{\chi}$. Furthermore, by the result in~\cite{2013advAllChromatic}, we know that $\delta_{\chi}$ takes only values of the form $\{\frac{s-3}{s-2},\frac{2s-5}{2s-3},\frac{s-2}{s-1}: s\in\mathbb{N}\}$ and in particular every number in $[0,1]$ is a jump; whereas~\cref{thm:SharpThreshold} shows that $0$ is an accumulation point for $\delta_{B}$.

Let us remark why blowup thresholds are more natural to consider than homomorphism thresholds. On one hand, it is in a sense easier to establish the lower bound: rather than proving the non-existence of homomorphic images with restrictions for $\delta_{\textup{hom}}$, one only needs to find a \emph{twin-free} construction for $\delta_{\textup{B}}$. Here, twin-free means no two vertices have the same neighborhood. On the other hand, the upper bound problem for $\delta_{\textup{B}}$ connects to current trends on sufficient conditions guaranteeing blowup structure. Such problem has been investigated by several recent works originating from different directions. For instance, the regularity lemma for semialgebraic hypergraphs~\cite{2005JCTASemi,2012Crelle} shows that such hypergraphs are blowup-like, i.e.,~most of the pairs in a regular partition are either complete or empty. Moreover, \cite{2024BlowupVCdensity,2024GraphToGeom} study various density conditions that guarantee the existence of large blowups.

\subsection{Our approach and future directions}

For~\cref{thm:KsKt}, to lower bound $\delta_{\textup{hom}}^{\textup{VC}}(K_s;K_t)$, it boils down to constructing triangle-free graphs with bounded VC-dimension and optimal density such
that all their bounded size homomorphic images have large clique number (see~\cref{lemma: lower bound for VC hom where s=3}). To construct such graphs, we start with two basic building blocks with bounded VC-dimension and then take certain iterative product of their blowups. The challenging part in this construction is to keep the triangle-freeness of the graph while increasing the clique number of all its homomorphic images when taking the iterative product.

To upper bound $\delta_{\textup{hom}}^{\textup{VC}}(K_s;K_t)$, we refine a partiton obtained based on Haussler packing lemma~\cite{1995PackingLemma} and show that this refinement yields a desired bounded size homomorphic image with given clique number. To bound the clique number of this homomophic image, we reduce the problem to the following result regarding codes on graphs, which we believe is of independent interest.

\begin{theorem}\label{thm:transform to vectors}
Let $r,s\in\mathbb{N}$ with \( r , s \geq 3 \)  and $t=r+s-3$. Let \( G \) be an \( n \)-vertex graph with $\delta(G) \geq \frac{(s - 3)(r - 1) + 1}{(s - 2)(r - 1) + 1} n$ and \( \boldsymbol{v}: V(G) \to \mathbb{F}_2^t \). If for any copy of \( K_{s - 2} \) in \( G \) with vertex set \( \{ x_1, x_2, \ldots, x_{s - 2} \} \), 
    \begin{equation}\label{eq:onlyone}
        \left| \{ j \in [t]: \boldsymbol{v}(x_1)_j = \boldsymbol{v}(x_2)_j = \cdots = \boldsymbol{v}(x_{s - 2})_j = 1 \} \right| \leq 1,
    \end{equation}
then $M := \sum_{x \in V(H)} \| \boldsymbol{v}(x) \| \leq \frac{(s - 3)(r - 1) + 1}{(s - 2)(r - 1) + 1} \cdot tn$, where \( \| \boldsymbol{v}(x) \| = \sum_{i = 1}^{t} \boldsymbol{v}(x)_i \) is the \( \ell_1 \)-norm.
\end{theorem}

Consider \( \boldsymbol{v}: V(G) \to \mathbb{F}_2^t \) as a code on the graph $G$. Note that when $s=4$, condition~\eqref{eq:onlyone} on~\( \boldsymbol{v}\) translates to the condition that any two codewords in~\( \boldsymbol{v}\) have Hamming distance at least $2t-2$ apart when their corresponding vertices are adjacent.

\medskip

For~\cref{thm:SharpThreshold}, the lower bound can be obtained using the so-called Andrasfai graphs. For the upper bound, our approach differs from previous work. For instance, Letzter and Snyder~\cite{2019JGTC3C5} used detailed structural analysis for $C_{5}$-free graphs. Under the same assumptions in~\cref{thm:SharpThreshold}, Ebsen and Schacht~\cite{2020COMBHomoOddCycle} employed probabilistic arguments to show that $G$ has a bounded $C_{2k-1}$-free homomorphic image; this argument however is not enough to show that $G$ is a blowup. Our proof avoids complicated structural discussions by employing tools from the theory of VC-dimension; it is furthermore algorithmic and finds explicitly the graph $H$. 

A key ingredient in our proof is the following. It shows that the VC-dimension of graphs satisfying the conditions in~\cref{thm:SharpThreshold} is bounded. To obtain the desired homomorphic image for~\cref{thm:SharpThreshold}, we iteratively refine the partition provided by~\cref{lemma:Bounded VC-dimension} and~\cref{lemma:Partition} (see~\cref{lemma:main}).

\begin{theorem}\label{lemma:Bounded VC-dimension}
     For integer $k\ge 2$ and any $\varepsilon>0$, let $G$ be an $n$-vertex maximal $C_{2k-1}$-free graph with minimum degree at least $(\frac{1}{2k-1}+\varepsilon)n$.
     Then $\textup{VC}(G)\le(2k-1)^3\binom{2k-2}{k-1}\cdot 2^{(2k-1)^{2k-2}}$.
\end{theorem}

The proof of~\cref{lemma:Bounded VC-dimension} utilizes Bollobas set-pair inequality~\cite{1965BollobasSetpair} and Ramsey type argument. It would be interesting to extend~\cref{lemma:Bounded VC-dimension} to other forbidden subgraphs. We refer the readers to~\cref{sec:vc-threshold} for further discussion.  

\medskip
{\bf \noindent Notations.} For a vertex $u\in V(G)$, we will use $N_{G}(u)$ (or $N(u)$ if the subscript is clear) to denote the set of neighbors of $u$. For a subset $T\subseteq V(G)$, we will use $G[T]$ to denote the subgraph induced by $T$. For two disjoint subsets $A,B\subseteq V(G)$, we say that $A$ and $B$ are \emph{complete} to each other if every vertex of $A$ is adjacent to every vertex of $B$, and \emph{anti-complete} to each other if no vertex of $A$ is adjacent to any vertex of $B$.  We also say $(A,B)$ \emph{induces} an edge in $G$ if there exists an edge such that one of its endpoint belongs to $A$ and the other belongs to $B$. For integers $i<j$, we use $[i,j]$ to denote the set $\{i,i+1,\ldots,j\}$. For the sake of clarity of presentation, we omit floors and ceilings and treat large numbers as integers whenever this does not affect the argument. Throughout this paper, we always assume that $n$ is sufficiently large whenever this is needed.

\medskip
{\bf \noindent Structure of this paper.} 
In~\cref{sec:KsKt}, we give the lower bound construction for~\cref{thm:KsKt}. We prove~\cref{thm:transform to vectors} and derive from it the upper bound for~\cref{thm:KsKt} in~\cref{sec:upp-bd-KsKt}. All results concerning odd cycles, including~\cref{thm:ByproductOne,thm:SharpThreshold,lemma:Bounded VC-dimension} are proved in~\cref{sec:OddCycles}. Finally, we conclude with some discussions of related problems in~\cref{sec:ConcludingRmks}.

\section{Lower bound constructions for~\cref{thm:KsKt}}\label{sec:KsKt}
To prove the lower bound $\delta_{\textup{hom}}^{\textup{VC}}(K_s; K_t) \geq \frac{(s-3)(t-s+2)+1}{(s-2)(t-s+2)+1}$, we need to construct $n$-vertex $K_s$-free graphs with bounded VC dimension and minimum degree at least $\left( \frac{(s-3)(t-s+2)+1}{(s-2)(t-s+2)+1}- o(1) \right) n$ that do not admit a bounded size $K_t$-free homomorphic image.

The following lemma is the main result of this section. It establishes the lower bound for the case $s=3$, from which the general case follows.

\begin{lemma}\label{lemma: lower bound for VC hom where s=3}
    For any $t \geq 3$ and $\varepsilon > 0$, there exists $d = d(t, \varepsilon)$ such that for any $f \geq 1$, there exists an $n$-vertex $K_{3}$-free graph $G = G(t, \varepsilon, f)$ with VC dimension at most $d$, minimum degree at least $( \frac{1}{t} - \varepsilon ) n$, such that if there is a homomorphism from $G$ to a graph $H$ with $|H| \leq f$, then $H$ contains a copy of $K_t$. In other words, $\delta_{\textup{hom}}^{\textup{VC}}(K_3; K_t) \geq \frac{1}{t}$.
\end{lemma}

To see that~\cref{lemma: lower bound for VC hom where s=3} implies the general case $t\ge s\ge 3$, set $t' = t - s + 3$. Let $G'$ be an $n'$-vertex $K_3$-free graph obtained from~\cref{lemma: lower bound for VC hom where s=3} with $t_{\ref{lemma: lower bound for VC hom where s=3}}=t'$ such that it has VC dimension at most $d'=d'(t,\varepsilon)$ and minimum degree at least $\left( \frac{1}{t'} - \varepsilon \right) n'$, and any homomorphism from $G'$ to a graph $H$ with $|H| \leq f$ implies that $H$ contains a $K_{t'}$. Take a complete $(s-2)$-partite graph with parts $V_1, V_2, \ldots, V_{s-2}$, where $|V_1|=n'$ and $|V_2|= \cdots= |V_{s-2}|=\frac{t'- 1}{t'}\cdot n'=\frac{t-s+2}{t-s+3}\cdot n'$. Embed $G'$ into $V_1$. We claim that the resulting graph $G$ has all the desired properties. First of all, it is easy to see that $G$ is $K_s$-free and has VC dimension at most $d'+s-2$. Next, let $\varphi$ be an arbitrary homomorphism from $G$ to a graph $H$ with $|H| \leq f$. We shall see that $H$ contains a copy of $K_t$. Note that $\varphi$ must map $G'$ to a subgraph $H'$ of $H$ containing $K_{t'}$. Pick one vertex from each of the $s - 3$ parts $V_2, \dots, V_{s-2}$. The images of these $s - 3$ vertices in $H$ together with a copy of $K_{t'}$ in $H'$ form a copy of $K_t$ in $H$. It is left to check that $G$ has the desired minimum degree. The number of vertices of $G$, say $n$, satisfies 
$
n = n' + (s - 3) \cdot \frac{t' - 1}{t'}\cdot n' = \frac{(s - 2)(t' - 1)+1}{t'}\cdot n'.
$
Thus, the minimum degree of the graph \( G \) satisfies
\begin{align*}
    \delta(G) &\geq \min\left\{n - \frac{t' - 1}{t'} n', \, n -n' + \left(\frac{1}{t'} - \varepsilon\right)n'\right\}\\
     &\geq \left(\frac{(s - 3)(t' - 1) + 1}{(s - 2)(t' - 1) + 1} - \varepsilon\right)n 
= \left(\frac{(s - 3)(t - s + 2) + 1}{(s - 2)(t - s + 2) + 1} - \varepsilon\right)n.
\end{align*}

We now prove~\cref{lemma: lower bound for VC hom where s=3} by considering the cases of even \( t \) and odd \( t \) separately. Recall that $\delta_{\textup{hom}}^{\textup{VC}}(K_3;K_3)=\delta_{\textup{hom}}(K_3)=\frac{1}{3}$, thus we assume $t\ge 4$. As the constructions are involved, we begin with the first case \( t = 4 \) as a warm up.
\subsection{Proof of~\cref{lemma: lower bound for VC hom where s=3} for $t=4$}
Let \( d(4, \varepsilon) = 12 \). Choose \( m \) to be a sufficiently large integer such that \( m \geq \frac{1}{\varepsilon} \) and \( m > 4f \). 

\begin{definition}
Let \( G = G(4, \varepsilon) \) be a 4-partite graph with \( 4m \) vertices. The vertex set is defined as  
\[
V(G) = \bigcup_{\alpha=1}^4 V^{(\alpha)}, \quad \text{where ~} V^{(\alpha)} = \{v^{(\alpha)}_i : i \in [m]\} \text{ ~for } \alpha \in [4].
\]  
The edge set of \( G \) is defined as follows:  
\begin{itemize}
    \item \( v^{(1)}_i v^{(2)}_j \in E(G) \) if and only if \( i \neq j \).  
    \item \( v^{(3)}_i v^{(4)}_j \in E(G) \) if and only if \( i \neq j \).  
    \item For \( \alpha \in \{1, 2\} \) and \( \beta \in \{3, 4\} \), \( v^{(\alpha)}_i v^{(\beta)}_j \in E(G) \) if and only if \( i = j \).  
\end{itemize}
\end{definition}

It is straightforward to verify that \( G \) is \( K_3 \)-free and has a minimum degree of \( \delta(G) = \frac{|G|}{4} + 1 \). We will now demonstrate the following properties.
\begin{claim}
The following properties of $G$ hold.
\begin{enumerate}
        \item[$\textup{(1)}$] $\textup{VC}(G)\le 12$.
        \item[$\textup{(2)}$] If there exists a graph $H$ such that $G\xrightarrow{\textup{hom}}H$ with $|H|\leq f$, then $H$ contains a copy of $K_{4}$.
\end{enumerate}    
\end{claim}

\begin{poc}
For (1), to upper bound the VC-dimension, let \( S \) be the largest shattered set in \( G \). For any \( \alpha \in [4] \), suppose \( S \) contains four vertices \( \{v_{i}^{(\alpha)}, v_{j}^{(\alpha)}, v_{k}^{(\alpha)}, v_{\ell}^{(\alpha)}\} \). In this case, no vertex \( v \in V(G) \) can satisfy \( |N(v) \cap \{v_{i}^{(\alpha)}, v_{j}^{(\alpha)}, v_{k}^{(\alpha)}, v_{\ell}^{(\alpha)}\}| = 2 \). This contradicts the definition of \( S \) being a shattered set. Therefore, we must have \( |S \cap V^{(\alpha)}| \leq 3 \) for all \( \alpha \in [4] \). Summing over all parts, we obtain \( \textup{VC}(G) = |S| \leq 12 \).

For (2), if there exists a graph \( H \) such that there is a homomorphism \( \varphi: G \xrightarrow{\textup{hom}} H \) with \( |H| \leq f \), we can color \( V(G) \) with at most \( f \) colors as follows. For each \( v \in V(H) \), we assign the color \( v \) to the vertex set \( \varphi^{-1}(v) \subseteq V(G) \). For any \( \alpha \in [4] \), we call a vertex \( v \in V^{(\alpha)} \) \emph{bad} if it is the only vertex in \( V^{(\alpha)} \) with its assigned color; otherwise, \( v \) is \emph{good}. Clearly, the total number of bad vertices is at most \( 4f \). Since \( m > 4f \), we can find an index \( i \in [m] \) such that \( \{v_{i}^{(1)}, v_{i}^{(2)}, v_{i}^{(3)}, v_{i}^{(4)}\} \) are all good vertices. Assume that for \( \alpha \in [4] \), the color of \( v_{i}^{(\alpha)} \) is \( c_\alpha \).  

Next, observe that for \( \alpha \in \{1, 2\} \) and \( \beta \in \{3, 4\} \), \( v_{i}^{(\alpha)} v_{i}^{(\beta)} \in E(G) \), so there is an edge between \( c_\alpha \) and \( c_\beta \) in \( H \). Since \( v_{i}^{(1)} \) is not a bad vertex, there exists \( v_{j}^{(1)} \) with color \( c_1 \) where \( i \neq j \). Consequently, \( v_{j}^{(1)} v_{i}^{(2)} \) forms an edge in \( E(G) \) between colors \( c_1 \) and \( c_2 \). Similarly, since \( v_{i}^{(3)} \) is not a bad vertex, there exists \( v_{k}^{(3)} \) with color \( c_3 \) where \( i \neq k \). Thus, \( v_{k}^{(3)} v_{i}^{(4)} \) forms an edge in \( E(G) \) between colors \( c_3 \) and \( c_4 \).  As a result, we have identified four color classes \( c_1, c_2, c_3, c_4 \in V(H) \) such that there is an edge in \( E(G) \) between any two of them. Since $\varphi$ is a homomorphism, \( c_1, c_2, c_3, c_4 \in V(H) \) form a copy of \( K_4 \) in \( H \).
\end{poc}

\subsection{Proof of~\cref{lemma: lower bound for VC hom where s=3} for $t=2k$}

Set \( d(2k, \varepsilon) = 4k^2 \). Let \( m \) be a sufficiently large integer such that \( \frac{m-1}{2km} \geq \frac{1}{2k} - \varepsilon \) and \( m > 2k(k-1)f \). The previous construction for \( t = 4 \) is a special case of the more general construction below.
\begin{definition}\label{def:t=2k}
Let $G=G(2k,\varepsilon)$ be a $2k$-partite graph with $2km^{k-1}$ vertices. The vertex set is defined as
\begin{equation*}
    V(G)=\bigcup_{\alpha=1}^{2k} V^{(\alpha)}, \text{~where\ ~}V^{(\alpha)}=\{v^{(\alpha)}_{\Vec{x}}:\Vec{x}=(x_1,x_2,\dots,x_{k-1})\in [m]^{k-1}\}\ \text{~for\ }\alpha\in [2k].
\end{equation*}
The edge set of \( G \) is defined as follows:
\begin{enumerate}
    \item For each \( \alpha \in [k] \), \( v^{(2\alpha - 1)}_{\Vec{x}} v^{(2\alpha)}_{\Vec{y}} \in E(G) \) if and only if \( x_1 \neq y_1 \).
    \item For each \( p \in [k-1] \), the following holds: For any \( \alpha \in \{1, 2, \dots, 2p\} \) and \( \beta \in \{2p+1, 2p+2\} \), \( v^{(\alpha)}_{\Vec{x}} v^{(\beta)}_{\Vec{y}} \in E(G) \) if and only if \( p \) is the maximum integer such that \( x_i = y_i \) for all \( i \in [p] \).
\end{enumerate}
\end{definition}

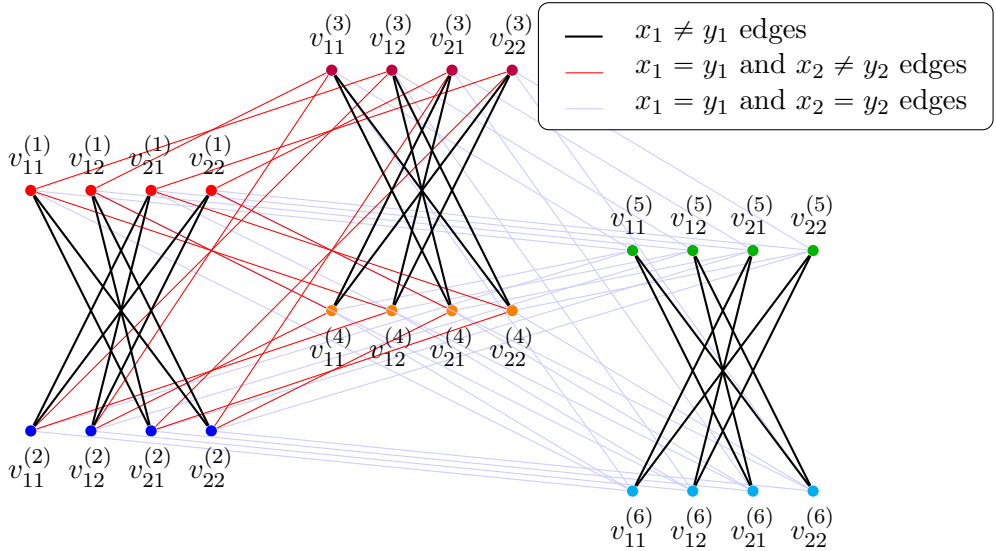
\begin{figure}[htbp]
\centering
\begin{tikzpicture}[
    node distance=8mm and 2mm,
    basic_node/.style={circle, fill, inner sep=1.5pt},
    partition1/.style={basic_node, red},
    partition2/.style={basic_node, blue},
    partition3/.style={basic_node, purple},
    partition4/.style={basic_node, orange},
    partition5/.style={basic_node, green!70!black},
    partition6/.style={basic_node, cyan},
    edge3/.style={blue!20},
    edge2/.style={red},
    edge1/.style={black, thick},
    scale=0.8
]

\foreach \x [count=\i] in {11,12,21,22} {
    \node[partition1] (v1-\x) at (-3 + \i, 1) {};
}

\foreach \x [count=\i] in {11,12,21,22} {
    \node[partition2] (v2-\x) at (-3 + \i, -3) {};
}

\foreach \x [count=\i] in {11,12,21,22} {
    \node[partition3] (v3-\x) at (2 + \i, 3) {};
}

\foreach \x [count=\i] in {11,12,21,22} {
    \node[partition4] (v4-\x) at (2 + \i, -1) {};
}

\foreach \x [count=\i] in {11,12,21,22} {
    \node[partition5] (v5-\x) at (7 + \i, 0) {}; 
}

\foreach \x [count=\i] in {11,12,21,22} {
    \node[partition6] (v6-\x) at (7 + \i, -4) {}; 
}

\foreach \A in {1,2,3,4} {
    \foreach \B in {5,6} {
        \foreach \x in {11,12,21,22} {
            \foreach \y in {11,12,21,22} {
                \ifnum \numexpr (\x/10) \relax = \numexpr (\y/10) \relax
                    \ifnum \numexpr (\x-10*(\x/10)) \relax = \numexpr (\y-10*(\y/10)) \relax
                        \draw[edge3] (v\A-\x) -- (v\B-\y);
                    \fi
                \fi
            }
        }
    }
}

\foreach \A in {1,2} {
    \foreach \B in {3,4} {
        \foreach \x in {11,12,21,22} {
            \foreach \y in {11,12,21,22} {
                \ifnum \numexpr (\x/10) \relax = \numexpr (\y/10) \relax
                    \ifnum \numexpr (\x-10*(\x/10)) \relax = \numexpr (\y-10*(\y/10)) \relax \else
                        \draw[edge2] (v\A-\x) -- (v\B-\y);
                    \fi
                \fi
            }
        }
    }
}

\foreach \A/\B in {1/2,3/4,5/6} {
    \foreach \x in {11,12,21,22} {
        \foreach \y in {11,12,21,22} {
            \ifnum \numexpr (\x/10) \relax = \numexpr (\y/10) \relax \else
                \draw[edge1] (v\A-\x) -- (v\B-\y);
            \fi
        }
    }
}

\foreach \x [count=\i] in {11,12,21,22} {
    \node[above=0mm of v1-\x] {$v^{(1)}_{\x}$};
}

\foreach \x [count=\i] in {11,12,21,22} {
    \node[below=0mm of v2-\x] {$v^{(2)}_{\x}$};
}

\foreach \x [count=\i] in {11,12,21,22} {
    \node[above=0mm of v3-\x] {$v^{(3)}_{\x}$};
}

\foreach \x [count=\i] in {11,12,21,22} {
    \node[below=0mm of v4-\x] {$v^{(4)}_{\x}$};
}

\foreach \x [count=\i] in {11,12,21,22} {
    \node[above=0mm of v5-\x] {$v^{(5)}_{\x}$};
}

\foreach \x [count=\i] in {11,12,21,22} {
    \node[below=0mm of v6-\x] {$v^{(6)}_{\x}$};
}

\node[draw, fill=white, rounded corners, anchor=south east] at (14,2) {
    \begin{tabular}{rl}
        \tikz{\draw[edge1] (0,0) -- (0.5,0);} & $x_1 \neq y_1$ edges \\
        \tikz{\draw[edge2] (0,0) -- (0.5,0);} & $x_1=y_1$ and $x_2 \neq y_2$ edges \\
        \tikz{\draw[edge3] (0,0) -- (0.5,0);} & $x_1=y_1$ and $x_2=y_2$ edges \\
    \end{tabular}
};

\end{tikzpicture}
\caption{Structure of the 6-partite graph $G$ with $t=6$, $m=2$}
\label{fig:6partite_graph}
\end{figure}

We then show some desired properties of the graph $G$ defined in~\cref{def:t=2k}.
\begin{prop}\label{Prop:K3Ktt=2k}
The following properties of $G$ holds:
\begin{enumerate}
        \item[$\textup{(1).}$] The minimum degree satisfies $\delta(G)\geq (m-1)m^{k-2}=\frac{m-1}{2km}\cdot |V(G)|\geq (\frac{1}{2k}-\varepsilon)\cdot |V(G)|$.
         \item[$\textup{(2).}$] $G$ is triangle free. 
        \item[$\textup{(3).}$] $\textup{VC}(G)\leq 4k^2$.
        \item[$\textup{(4).}$] If there exists a graph $H$ such that $G\xrightarrow{\textup{hom}}H$ with $|H|\leq f$, then $H$ contains a copy of $K_{2k}$.
    \end{enumerate}
\end{prop}

\begin{proof}
For (1), note that for a vertex \( v = v^{(2\alpha - 1)}_{\Vec{x}} \in V(G) \), $\alpha\in[k]$, there are at least \( (m-1)m^{k-2} \) vertices of the form \( v^{(2\alpha)}_{\Vec{y}} \) where \( x_1 \neq y_1 \), so \( v \) has at least \( (m-1)m^{k-2} \) neighbors. The same bound holds for vertices in $V^{(2\alpha)}$. Thus, the minimum degree is at least \( (m-1)m^{k-2} \).
     
     For (2), suppose that \( v^{(\alpha)}_{\Vec{x}}v^{(\beta)}_{\Vec{y}}v^{(\gamma)}_{\Vec{z}} \) form a triangle in \( G \). Then, \( \alpha, \beta, \gamma \) are pairwise distinct. If there exists a \( p_{0} \) such that \( \alpha = 2p_{0} - 1 \) and \( \beta = 2p_{0} \). In this case, \( \gamma \notin \{2p_{0} - 1, 2p_{0}\} \), so we have \( x_1 \neq y_1 = z_1 = x_1 \), which is a contradiction.

Thus, without loss of generality, we may assume that no two of \( \alpha \), \( \beta \), and \( \gamma \) belong to the same block, where the blocks are defined as consecutive pairs of integers \( \{2p - 1, 2p\} \) for \( p \in [k] \). Let \( p_{xy} \) be the maximum integer \( p \) such that \( x_i = y_i \) for all \( i \in [p] \), and define \( p_{yz} \) and \( p_{xz} \) similarly. Without loss of generality, assume \( p_{xy} \leq p_{yz} \leq p_{xz} \). Therefore, for \( i \in [p_{yz}] \subseteq [p_{xz}] \), we have \( y_i = z_i = x_i \), implying that \( p_{xy} \geq p_{yz} \). Hence, \( p_{xy} = p_{yz} \).

Let \( p = p_{xy} = p_{yz} \). If \( \beta \in \{1, 2, \dots, 2p\} \), then \( \alpha,\gamma \in \{2p + 1, 2p + 2\} \), which is a contradiction, since no two of \( \alpha \), \( \beta \), and \( \gamma \) belong to the same block. Therefore, \( \beta \in \{2p + 1, 2p + 2\} \) and \( \alpha,\gamma \in \{1, 2, \dots, 2p\} \). By definition, one of \( \alpha \) and \( \gamma \) must be in \( \{2p_{xz} + 1, 2p_{xz} + 2\} \). Thus, we have \( p_{xz} < p = p_{xy} \), which results in a contradiction. This completes the proof.
    
    For (3), by definition, the event that \( v^{(\alpha)}_{\Vec{x}} v^{(\beta)}_{\Vec{y}} \) is an edge in \( G \) is entirely determined by the values of \( \alpha \), \( \beta \), and whether the conditions \( x_1 = y_1 \), \( x_2 = y_2 \), \(\dots\), and \( x_{k-1} = y_{k-1} \) hold. To upper bound the VC-dimension, let \( S \) denote the largest shattered set in \( G \). 

For any vertex \( v^{(\alpha)}_{\Vec{x}} \), the number of possible connection types to the set \( S \) is at most \( 2k(|S| + 1)^{k-1} \). This accounts for \( 2k \) choices for \( \alpha \) and \( |S| + 1 \) choices for each coordinate of \( \Vec{x} \). Since \( S \) is a shattered set, we have the inequality \( 2^{|S|} \leq 2k(|S| + 1)^{k-1} \). Given that \( k \geq 2 \), we can conclude that the VC-dimension of \( G \) satisfies \( \VC(G) = |S| \leq 4k^2 \).
    
    For (4), if there exists a graph \( H \) such that \( G \xrightarrow{\textup{hom}} H \) with \( |H| \leq f \), we color \( V(G) \) using at most \( f \) colors by the images as before. For any \( \alpha \in [2k] \), we say a vertex \( v^{(\alpha)}_{\Vec{x}} \) is \emph{bad} if it receives some color \( c \) and there exists an integer \( p \in [k-1] \) such that there is no vertex of color \( c \) in the set
\[ \{ v^{(\alpha)}_{\Vec{y}} : y_p \neq x_p \text{ and for all } i \in [p-1], y_i = x_i \}. \] Otherwise, we say \( v^{(\alpha)}_{\Vec{x}} \) is \emph{good}. Similar to the case of \( t = 4 \), we need to count the number of bad vertices. For a fixed color \( c \), \( \alpha \), and \( p \), for any \( \Vec{x} \), there is at most one vertex in \[ \{ v^{(\alpha)}_{\Vec{y}} : \text{ for all } i \in [k-1] \setminus \{p\}, y_i = x_i \} \] that is a bad vertex with respect to \( c \), \( \alpha \), and \( p \). Thus, the total number of bad vertices in \( G \) is at most
\[
\sum_{c \in [f]} \sum_{\alpha \in [2k]} \sum_{p \in [k-1]} m^{p-1} \cdot 1 \cdot m^{k-1-p} = 2k(k-1)fm^{k-2} < m^{k-1}.
\]
Therefore, we can find an \( \Vec{x} \in [m]^{k-1} \) such that \( \{ v_{\Vec{x}}^{(1)}, v_{\Vec{x}}^{(2)}, \dots, v_{\Vec{x}}^{(2k)} \} \) are all good vertices. Fix this \( \Vec{x} \), and assume that for any \( \alpha \in [2k] \), the color of \( v_{\Vec{x}}^{(\alpha)} \) is \( c_\alpha \). It suffices to show that for any \( 1 \leq \alpha < \beta \leq 2k \), there is an edge between the vertices of color \( c_\alpha \) and \( c_\beta \). We consider the following three cases, which cover all possible scenarios.
    \begin{itemize}
        \item For \( p \in [k] \) and \( \alpha = 2p - 1 \), \( \beta = 2p \), since \( v_{\Vec{x}}^{(2p)} \) is a good vertex, there exists a vertex \( v^{(2p)}_{\Vec{y}'} \) in \( \{ v^{(2p)}_{\Vec{y}} : y_1 \neq x_1 \} \) with color \( c_\beta \). Then, \( v_{\Vec{x}}^{(2p-1)} v^{(2p)}_{\Vec{y}'} \) is an edge between the vertices of colors \( c_\alpha \) and \( c_\beta \).
        \item 
For \( p \in [k-2] \), \( \alpha \in \{ 1, 2, \dots, 2p \} \) and \( \beta \in \{ 2p+1, 2p+2 \} \), since \( v_{\Vec{x}}^{(\beta)} \) is a good vertex, there exists a vertex \( v^{(\beta)}_{\Vec{y}'} \) in \( \{ v^{(\beta)}_{\Vec{y}} : y_{p+1} \neq x_{p+1} \text{ and for } i \in [p], y_i = x_i \} \) with color \( c_\beta \). Then, \( v_{\Vec{x}}^{(\alpha)} v^{(\beta)}_{\Vec{y}'} \) is an edge between the vertices of colors \( c_\alpha \) and \( c_\beta \).
        \item 
For \( p = k-1 \), where \( \alpha \in \{ 1, 2, \dots, 2p \} \) and \( \beta \in \{ 2p+1, 2p+2 \} \), \( v_{\Vec{x}}^{(\alpha)} v^{(\beta)}_{\Vec{x}} \) is an edge between the vertices of colors \( c_\alpha \) and \( c_\beta \).
    \end{itemize}
Thus, \( c_1, c_2, \dots, c_{2k} \) form a copy of \( K_{2k} \) in \( H \). This finishes the proof.
\end{proof}

\subsection{Proof of~\cref{lemma: lower bound for VC hom where s=3} for $t=2k+1$}
Define \( d(2k+1, \varepsilon) = 4k^2 \). Let \( m \) be a sufficiently large number such that \( \frac{m-1}{(2k+1)m} \geq \frac{1}{2k+1} - \varepsilon \) and \( m > (2k+1)(k-1)f + 3f \). The odd case is similar to the even case \( t = 2k \), but with a slight modification: the last two parts, \( 2k-1 \) and \( 2k \), are adjusted to three parts, \( 2k-1 \), \( 2k \), and \( 2k+1 \).
\begin{definition}
Let $G=G(2k+1,\varepsilon)$ be a $(2k+1)$-partite graph with $(2k+1)m^{k-1}$ vertices, where 
\begin{equation*}
    V(G)=\bigcup_{\alpha=1}^{2k+1} V^{(\alpha)},\ \text{where\ }V^{(\alpha)}=\{v^{(\alpha)}_{\Vec{x}}:\Vec{x}=(x_1,x_2,\dots,x_{k-1})\in [m]^{k-1}\}\ \text{for\ }\alpha\in [2k+1].
\end{equation*}
The edge set of $G$ is defined as follows.
\begin{enumerate}
    \item For each $\alpha\in [k-1]$, $v^{(2\alpha -1)}_{\Vec{x}}v^{(2\alpha )}_{\Vec{y}}\in E(G)$ if and only if $x_1\neq y_1$.
    \item  
    For each $p\in [k-2]$ the following holds: For any $\alpha\in \{1,2,\dots,2p\}$, $\beta\in \{2p+1,2p+2\}$, $v^{(\alpha)}_{\Vec{x}}v^{(\beta)}_{\Vec{y}}\in E(G)$ if and only if $p$ is the largest integer such that $x_i=y_i$ holds for all $i\in [p]$.

    \item For any $\alpha\in \{1,2,\dots,2k-2\}$, $\beta\in \{2k-1,2k,2k+1\}$, $v^{(\alpha)}_{\Vec{x}}v^{(\beta)}_{\Vec{y}}\in E(G)$ if and only if $x_i=y_i$ for all $i\in [k-1]$.

    \item For $(\alpha,\beta)=(2k-1,2k)$, $(2k,2k+1)$ or $(2k+1,2k-1)$, $v^{(\alpha)}_{\Vec{x}}v^{(\beta)}_{\Vec{y}}\in E(G)$ if and only if $x_1<y_1$.
\end{enumerate}
\end{definition}

\begin{figure}[htbp]
\centering
\begin{tikzpicture}[
    node distance=8mm and 2mm,
    basic_node/.style={circle, fill, inner sep=1.5pt},
    partition1/.style={basic_node, red},
    partition2/.style={basic_node, blue},
    partition3/.style={basic_node, green!70!black},
    partition4/.style={basic_node, orange},
    partition5/.style={basic_node, purple},
    rule1/.style={black, thick},
    rule3/.style={blue!45},
    rule4/.style={red, thick},
    scale=0.8
]

\foreach \i in {1,...,4} {
    \node[partition1] (v1-\i) at (-5 + \i, 3) {}; 
    \node[above=1mm of v1-\i] {$v^{(1)}_{\i}$};
}

\foreach \i in {1,...,4} {
    \node[partition2] (v2-\i) at (-5 + \i, -1) {}; 
    \node[below=1mm of v2-\i] {$v^{(2)}_{\i}$};
}

\foreach \i/\angle in {1/0,2/30,3/60,4/90} {
    \node[partition3] (v3-\i) at ({5+3*cos(\angle)},{3*sin(\angle)}) {};
    \node[above=0mm of v3-\i] {$v^{(3)}_{\i}$};
}

\foreach \i/\angle in {1/120,2/150,3/180,4/210} {
    \node[partition4] (v4-\i) at ({5+3*cos(\angle)},{3*sin(\angle)}) {};
    \node[above=0mm of v4-\i] {$v^{(4)}_{\i}$};
}

\foreach \i/\angle in {1/240,2/270,3/300,4/330} {
    \node[partition5] (v5-\i) at ({5+3*cos(\angle)},{3*sin(\angle)}) {};
    \node[below=0mm of v5-\i] {$v^{(5)}_{\i}$};
}

\foreach \source in {v1,v2} {  
    \foreach \i in {1,...,4} {
        \foreach \target in {v3,v4,v5} {
            \draw[rule3] (\source-\i) -- (\target-\i);
        }
    }
}

\foreach \i in {1,...,4} {
    \foreach \j in {1,...,4} {
        \ifnum\i=\j\else
            \draw[rule1] (v1-\i) -- (v2-\j);
        \fi
    }
}

\foreach \i in {1,...,3} {
    \pgfmathtruncatemacro{\jstart}{\i + 1} 
    \foreach \j in {\jstart,...,4} {
        \draw[rule4] (v3-\i) -- (v4-\j);
        \draw[rule4] (v4-\i) -- (v5-\j);
        \draw[rule4] (v5-\i) -- (v3-\j);
    }
}

\node[draw, fill=white, rounded corners, anchor=south east] at (1,-4.5) {
    \begin{tabular}{rl}
        \tikz{\draw[rule1] (0,0)--(0.5,0);} & $x_1 \neq y_1$ edges \\
        \tikz{\draw[rule4] (0,0)--(0.5,0);} & $x_1 < y_1$ edges \\
        \tikz{\draw[rule3] (0,0)--(0.5,0);} & $x_i = y_i$ for all $i\in[1]$ edges \\
    \end{tabular}
};

\end{tikzpicture}
\caption{Structure of the 5-partite graph $G$ with $t=5$, $m=4$}
\label{fig:5partite_adjusted}
\end{figure}

\begin{prop}\label{Prop:K3Ktt=2k+1}
The following properties of $G$ holds:
\begin{enumerate}
        \item[$\textup{(1)}$] The minimum degree $\delta(G)\geq (m-1)m^{k-2}=\frac{m-1}{(2k+1)m}\cdot|V(G)|\geq (\frac{1}{2k+1}-\varepsilon)\cdot |V(G)|$.
         \item[$\textup{(2)}$] $G$ is triangle free. 
        \item[$\textup{(3)}$] $\textup{VC}(G)\leq 4k^2$.
        \item[$\textup{(4)}$] If there exists a graph $H$ such that $G\xrightarrow{\textup{hom}}H$ with $|H|\leq f$, then $H$ contains a copy of $K_{2k+1}$.
    \end{enumerate}
\end{prop}

\begin{proof}
The proof follows a similar approach to the proof of~\cref{Prop:K3Ktt=2k}, we streamline the key differences.

For (1), it suffices to verify that for any \(v \in V^{(2k-1)} \cup V^{(2k)} \cup V^{(2k+1)}\), \(v\) has at least \((m-1)m^{k-2}\) neighbors. Without loss of generality, assume \(v = v^{(2k-1)}_{\Vec{x}} \in V^{(2k-1)}\). The neighborhood of \(v\) includes the set \(\{v^{(2k)}_{\Vec{y}} : x_1 < y_1\} \cup \{v^{(2k+1)}_{\Vec{y}} : x_1 > y_1\}\), whose size is exactly \((m-1)m^{k-2}\).

For (2), we now define the blocks as \(\{2p-1, 2p\}\) for \(p \in [k-1]\) and the set \(\{2k-1, 2k, 2k+1\}\). The remainder of the proof is identical to that of~\cref{Prop:K3Ktt=2k}~(2), provided that \(\{2k-1, 2k, 2k+1\}\) is treated as a single block. It suffices to verify that no two of \(\alpha\), \(\beta\), and \(\gamma\) belong to the block \(\{2k-1, 2k, 2k+1\}\).  

We proceed by contradiction. Assume that two of \(\alpha\), \(\beta\), and \(\gamma\) belong to the block \(\{2k-1, 2k, 2k+1\}\). Since \(\alpha\), \(\beta\), and \(\gamma\) are pairwise distinct, without loss of generality, let \(\alpha = 2k-1\) and \(\beta = 2k\). Recall that \(v^{(\alpha)}_{\Vec{x}}v^{(\beta)}_{\Vec{y}}v^{(\gamma)}_{\Vec{z}}\) form a triangle in \(G\). If \(\gamma = 2k+1\), then \(x_1 < y_1 < z_1 < x_1\), which is a contradiction. If \(\gamma \in [2k-2]\), then \(x_1 < y_1 = z_1 = x_1\), which also leads to a contradiction.
    
    For (3), by the same argument in the proof of~\cref{Prop:K3Ktt=2k}~(3), $2^{|S|}\leq (2k+1)(|S|+1)^{k-1}$. Notice that $k\geq 2$, therefore $\VC(G)=|S|\leq 4k^2$.
    
    For (4), color \(V(G)\) with at most \(f\) colors by the images. For any \(\alpha \in [2k+1]\), we call a vertex \(v^{(\alpha)}_{\Vec{x}}\) \emph{bad} if it is assigned color \(c\) and there exists an integer \(p \in [k-1]\) such that no vertex of color \(c\) exists in the set \[\{v^{(\alpha)}_{\Vec{y}} : y_{p} \neq x_{p} \text{ and } y_i = x_i \text{ for all } i \leq p-1\}.\]
    For any \( \alpha \in \{2k-1, 2k, 2k+1\} \), we say a vertex \( v^{(\alpha)}_{\Vec{x}} \) is \emph{cyclic-bad} if it has color \( c \) and there is no vertex of color \( c \) in the set \( \{v^{(\alpha)}_{\Vec{y}} : x_1 < y_1\} \). A vertex \( v \in V(G) \) is called \emph{good} if it is neither bad nor cyclic-bad. As in the proof of~\cref{Prop:K3Ktt=2k}~(4), the number of bad vertices in \( G \) is at most  
\[
\sum_{c \in [f]} \sum_{\alpha \in [2k+1]} \sum_{p \in [k-1]} m^{p-1} \cdot 1 \cdot m^{k-1-p} = (2k+1)(k-1)f m^{k-2}.
\]
For fixed \( \alpha \) and color \( c \), for any \( \Vec{x} \), there is at most one cyclic-bad vertex in the set \[ \{v^{(\alpha)}_{\Vec{y}} : \text{for all } i \in [k-1] \setminus \{1\}, \, y_i = x_i\}. \] Hence, the total number of cyclic-bad vertices in \( G \) is at most  
\[
\sum_{c \in [f]} \sum_{\alpha \in \{2k-1, 2k, 2k+1\}} m^{k-2} = 3f m^{k-2}.
\] Therefore the number of non-good vertices is at most $(2k+1)(k-1)fm^{k-2}+3fm^{k-2}<m^{k-1}$.
Thus, we can find an \( \Vec{x} \in [m]^{k-1} \) such that the vertices \( \{v_{\Vec{x}}^{(1)}, v_{\Vec{x}}^{(2)}, \dots, v_{\Vec{x}}^{(2k)}, v_{\Vec{x}}^{(2k+1)}\} \) are all good. Let \( c_\alpha \) denote the color of \( v_{\Vec{x}}^{(\alpha)} \) for each \( \alpha \in [2k+1] \). To establish that for any \( \alpha \neq \beta \), where \( \alpha, \beta \in [2k+1] \), there is an edge between the vertices of colors \( c_\alpha \) and \( c_\beta \), it suffices to consider the cases \( (\alpha, \beta) \in \{(2k-1, 2k), (2k, 2k+1), (2k+1, 2k-1)\} \), since all other cases follow similarly to the proof of~\cref{Prop:K3Ktt=2k}~(4). Since \( v_{\Vec{x}}^{(\beta)} \) is a good vertex, there exists a vertex \( v^{(\beta)}_{\Vec{y}'} \) of color \( c_\beta \) in the set \( \{v^{(\beta)}_{\Vec{y}} : x_1 < y_1\} \). Consequently, \( v_{\Vec{x}}^{(\alpha)} v^{(\beta)}_{\Vec{y}'} \) forms an edge between the vertices of colors \( c_\alpha \) and \( c_\beta \).
Therefore, \( c_1, c_2, \dots, c_{2k+1} \) form a copy of \( K_{2k+1} \) in \( F \). This finishes the proof.
\end{proof}

\section{Upper bound of~\cref{thm:KsKt}}\label{sec:upp-bd-KsKt}
We first show how to obtain the upper bound of~\cref{thm:KsKt} via~\cref{thm:transform to vectors}.
We need a partition lemma which can be derived from a lemma of Haussler~\cite{1995PackingLemma} that upper bounds the size of a packing in a set system with bounded VC-dimension.

\begin{lemma}[Partition lemma~\cite{2024GraphToGeom}]\label{lemma:Partition}
  Let $d$ be a positive integer and $G$ be an $n$-vertex graph with VC-dimension at most $d$. Let $\mathcal{F}:=\{N_{G}(v):v\in V(G)\}$. 
  Then for any $1\le a\le n$, there is a partition $V(G)=V_{1}\sqcup V_{2}\sqcup\cdots\sqcup V_{r}$, where $r=e(d+1)\cdot (2e)^{d}\big(\frac{n}{a}\big)^{d}$, such that for each $i\in [r]$, any pair of vertices $u,v\in V_{i}$ satisfies that $|N_{G}(v)\triangle N_{G}(u)|\le 2a$. 
\end{lemma}

Throughout this section, let \( d \geq 1 \) and $t\ge s\ge 3$ be positive integers and
set
$$r := t - s + 3 \geq 3 \quad \text{ and } \quad \beta_j:=(s-2-j)(r-1)+1, \text{ for } j=0,1,\ldots,s-2.$$
So, $t=s+r-3$ and $\frac{(s-3)(t-s+2)+1}{(s-2)(t-s+2)+1}=\frac{\beta_1}{\beta_0}$. 
To prove the upper bound in~\cref{thm:KsKt}, we need to show the following. Let \( \varepsilon > 0 \) and \( G \) be an \( n \)-vertex \( K_s \)-free graph with VC-dimension at most \( d \). If \( \delta(G) \geq (\frac{\beta_1}{\beta_0} + \varepsilon)n \), then \( G \) admits a homomorphism to a \( K_{s+r-3} \)-free graph \( H \) with \( |H| = O_{\varepsilon,d,r}(1) \).

As $\delta(G)\ge (\frac{s-3}{s-2}+\varepsilon)n$ and $\delta_{\chi}^{\textup{VC}}(K_s)=\frac{s-3}{s-2}$, $G$ can be colored by a bounded number of colors, say $C$ colors. Then, apply~\cref{lemma:Partition} to refine the color classes of this coloring to obtain a partition \( V(G) = U_{1} \sqcup \cdots \sqcup U_{K} \), where \( K \leq C \cdot e(d+1)(2e)^{d} \big(\frac{2C \cdot r!}{\varepsilon}\big)^{d} \), such that each \( U_{i} \) is an independent set, and for any pair of vertices \( u, v \in U_{i} \), $i\in[K]$, we have \( |N_{G}(u) \triangle N_{G}(v)| \leq \frac{\varepsilon n}{C \cdot r!} \). 

Define an auxiliary graph \( H = H(U_{1}, \ldots, U_{K}) \) with vertex set \( \{U_{1}, \ldots, U_{K}\} \), and \( U_{i}U_{j} \in E(H) \) if and only if there is an edge between $U_{i}$ and $U_{j}$ in $G$. Clearly, \( G \xrightarrow{\textup{hom}} H \). We shall show that \( H \) is our desired bounded size \( K_{s + r - 3} \)-free homomorphic image of $G$. To this end, we fix \( r \geq 3 \) and induct on $s\ge 3$.

For the base case \( s = 3 \), suppose to the contrary that there exists a copy of \( K_{r} \) in \( H \), with parts \( U_{i_{1}}, \ldots, U_{i_{r}} \). Take an arbitrary vertex from each part, denoted as \( v_{i_{j}} \in U_{i_{j}} \) for each \( j \in [r] \). Since when $s=3$, \( \delta(G) \geq \big(\frac{1}{r} + \varepsilon\big)n \), the pigeonhole principle ensures that there is a pair of vertices among \( \{v_{i_{1}}, \ldots, v_{i_{r}}\} \) with at least \( \frac{\varepsilon n}{r^{2}} \) common neighbors. Without loss of generality, assume \( v_{i_{1}} \) and \( v_{i_{2}} \) satisfy \( |N(v_{i_{1}}) \cap N(v_{i_{2}})| \geq \frac{\varepsilon n}{r^{2}} \). Since \( (U_{i_{1}}, U_{i_{2}}) \) induces an edge in \( G \), there exist vertices \( a \in U_{i_{1}} \) and \( b \in U_{i_{2}} \) such that \( ab \in E(G) \). By the definition of \( H \), we have \( |N(a) \cap N(b)| \geq \frac{\varepsilon n}{r^{2}} - 2 \cdot \frac{\varepsilon n}{C r!} \geq \frac{\varepsilon n}{2r^{2}} \). This implies that there is at least one vertex in \( N(a) \cap N(b) \), contradicting the assumption that \( G \) is triangle-free. This concludes the base case.

For the inductive step, fix $s\ge 4$. Suppose there is a copy of \( K_{s + r - 3} \) in \( H \). Without loss of generality, assume that \( U_{1}, U_{2}, \ldots, U_{s + r - 3} \) form a copy of \( K_{s + r - 3} \) in \( H \). As $U_1,\ldots, U_{s+r-3}$ form a clique, by the definition of $H$, for each distinct pair \( (i, j) \in \binom{[s + r - 3]}{2} \), we can select a vertex \( u_{j}^{i} \in U_{j} \) and a vertex \( u_{i}^{j} \in U_{i} \) such that \( u_{j}^{i} u_{i}^{j} \in E(G) \). The number of chosen vertices in each part is at most \( s + r - 4 \), and the total number of chosen vertices is at most \( 2 \binom{s + r - 3}{2} \). For each \( \ell \in [s + r - 3] \), define 
\[ N_{\ell} := \bigcap_{h \in [s + r - 3] \setminus \{\ell\}} N_{G}(u_{\ell}^{h})\]
to be the common neighbors of the chosen vertices in $U_{\ell}$. By the definition of \( H \), we observe that for each \( \ell \in [s + r - 3] \), $$|N_{\ell}| \geq \delta(G) - (s + r - 4) \cdot \frac{\varepsilon n}{C \cdot r!} \geq \left( \frac{\beta_1}{\beta_0} + \frac{ \varepsilon}{2} \right) n.$$ 

For each vertex $x\in V(G)$, we assign a binary vector $\boldsymbol{v}(x)=(\boldsymbol{v}(x)_{1},\ldots,\boldsymbol{v}(x)_{s+r-3})$ of length $s+r-3$, where for each $i\in [s+r-3]$, $\boldsymbol{v}(x)_{i}=1$ if and only if $x\in N_{\ell}$ and otherwise $\boldsymbol{v}(x)_{i}=0$. We claim that inequality~\eqref{eq:onlyone} in~\cref{thm:transform to vectors} holds. Indeed, suppose $x_{1},x_{2},\ldots,x_{s-2}\in V(G)$ form a copy of $K_{s-2}$, but
    \begin{equation*}
        |\{\ell\in[s+r-3]:\boldsymbol{v}(x_{1})_{\ell}=\boldsymbol{v}(x_{2})_{\ell}=\cdots=\boldsymbol{v}(x_{s-2})_{\ell}=1\}|\ge 2.
    \end{equation*}
Then there exist some $\ell_{1},\ell_{2}\in [s+r-3]$ such that $u_{\ell_{1}}^{\ell_{2}},u_{\ell_{2}}^{\ell_{1}},x_{1},\ldots,x_{s-2}$ form a copy of $K_{s}$ in $G$, a contradiction. Thus, by~\cref{thm:transform to vectors}, $M = \sum_{x \in V(G)} \| \boldsymbol{v}(x) \| \leq \frac{\beta_1}{\beta_0} (s+r-3)n$. 

Let us bound $M$ in another way. Consider the $n$-by-$(s+r-3)$ 0/1-matrix where $\boldsymbol{v}(x_{i})$ is the $i$-th row. Then $M$ is the number of $1$s in this matrix and $|N_{\ell}|$ is the column sum for the $\ell$-th column. Thus, $M = \sum_{\ell=1}^{s+r-3}|N_\ell|>  \frac{\beta_1}{\beta_0} (s+r-3)n$, a contradiction. 

This completes the proof of \Cref{thm:KsKt}.

\subsection{Proof of \Cref{thm:transform to vectors}}
  Suppose for the sake of contradiction that  
\begin{align}\label{ineq:contrary}  
    M = \sum_{x \in V(G)} \|\boldsymbol{v}(x)\| > \frac{(s-3)(r-1)+1}{(s-2)(r-1)+1} \cdot tn.  
\end{align}
Then we can find a `heavy' $K_{s-2}$ in $G$ as follows. 

\begin{claim}\label{prop:exist Ks-2 copy}
   There exist vertices \( x_0, \ldots, x_{s-3} \in V(G) \) that form a copy of \( K_{s-2} \) such that 
   \begin{align}\label{eq:sum=1+t(s-3)}  
    \sum_{j=0}^{s-3} \|\boldsymbol{v}(x_j)\| = 1 + t(s-3),  
\end{align}
    and for every $i\in \{0,1,\ldots,s-4\}$,
 \begin{align}\label{eq:max weight}
    \|\boldsymbol{v}(x_0)\|=\max\{\|\boldsymbol{v}(y)\| : y \in V(H)\}  \quad \text{and} \quad \|\boldsymbol{v}(x_{i+1})\|=
    \max\{\|\boldsymbol{v}(y)\| : y \in \bigcap_{0 \leq j \leq i} N_H(x_j)\}.
\end{align} 
\end{claim}

Let us first derive~\Cref{thm:transform to vectors}. Given $x_{0},x_{1},\ldots,x_{s-3}$ from~\cref{prop:exist Ks-2 copy}, we observe that vertices in their common neighborhood cannot have too large weight.
    \begin{claim}\label{claim:CommonNeighborSmallWeight}
   For each vertex $w\in \bigcap_{0\le h\le s-3} N(x_h)$, $\|\boldsymbol{v}(w)\|\le \binom{s-2}{s-3} = s-2$. 
\end{claim}
\begin{poc}
By~\eqref{eq:onlyone} and~\eqref{eq:sum=1+t(s-3)}, we observe that there must exist exactly one index \( p \in [t] \) such that  
\[
\boldsymbol{v}(x_0)_{p} = \cdots = \boldsymbol{v}(x_{s-3})_{p} = 1,
\]  
and for all other indices \( \ell \in [t] \setminus \{p\} \), we have  
\[
|\{ a \in \{0,1,\dots,s-3\} : \boldsymbol{v}(x_a)_{\ell}  = 1 \}| = s-3.
\]  

As the number of distinct binary vectors of length \( s-2 \) with exactly \( s-3 \) ones is  \( \binom{s-2}{s-3} = s-2 \), we can partition all indices \( \ell \in [t] \setminus \{p\} \) into \( s-2 \) blocks \( \mathcal{B}_0, \mathcal{B}_1, \dots, \mathcal{B}_{s-3} \), such that for any two indices \( \ell_1, \ell_2 \in\mathcal{B}_i\), $i\in \{0,1,\ldots, s-3\}$, we have  
\[
\boldsymbol{v}(x_a)_{\ell_1} = \boldsymbol{v}(x_a)_{\ell_2}=1 \quad\text{ if and only if }\quad a \in \{0,1,\dots,s-3\}\setminus\{i\}.
\]    

Therefore, if \( \|\boldsymbol{v}(w)\| \geq s-1 \) for some $w\in \bigcap_{0\le h\le s-3} N(x_h)$, then there must exist two distinct indices \( k_1, k_2 \in [t] \) such that  
\[
\boldsymbol{v}(w)_{k_1} = \boldsymbol{v}(w)_{k_2} = 1,
\]  
and \( k_1, k_2 \) belong to the same block, say $k_1,k_2\in\mathcal{B}_{0}$ (or one of \( k_1 \) and \( k_2 \) equals \( p \), and we consider the index $p$ belongs to every block). Then consider $w,x_1,\ldots,x_{s-3}$, which form a copy of $K_{s-2}$ and   
\[
k_1,k_2\in \{ q \in [t] \mid \boldsymbol{v}(w)_q = \boldsymbol{v}(x_1)_q = \cdots = \boldsymbol{v}(x_{s-3})_q = 1 \}.
\]  
This contradicts~\eqref{eq:onlyone}, completing the proof. 
\end{poc}

We can now bound $M$ to arrive at a contradiction. By~\eqref{eq:max weight},
\begin{align}\label{eq:bound-M-above}
    M \leq &(n - |N(x_0)|) \cdot \|\boldsymbol{v}(x_0)\|  
    + |N(x_0) \setminus N(x_1)| \cdot \|\boldsymbol{v}(x_1)\|  
    + \cdots \notag\\
    &\quad + |\big(\bigcap_{0 \leq h \leq s-4} N(x_h)\big) \setminus N(x_{s-3})| \cdot \|\boldsymbol{v}(x_{s-3})\|  
    + |\bigcap_{0 \leq h \leq s-3} N(x_h)| \cdot (s-2).
\end{align}
The right-hand side above is maximized when the terms  
\begin{equation*}  
    (n - |N(x_0)|), \, |N(x_0) \setminus N(x_1)|, \, \ldots, \, |\big(\bigcap_{0 \leq h \leq s-4} N(x_h)\big) \setminus N(x_{s-3})|  
\end{equation*}  
are sequentially maximized and note that each of these terms is at most $n-\delta(G)=\frac{(r-1)n}{(s-2)(r-1)+1}$. Upon noticing
\begin{equation}\label{eq:comm-neigh}
    |\bigcap_{0 \leq h \leq s-3} N(x_h)|\ge (s-2)\delta(G)-(s-3)n\ge \frac{n}{(s-2)(r-1)+1}
\end{equation}
and recalling~\eqref{eq:sum=1+t(s-3)}, we get that 
\begin{align}\label{eq:bound-M-above2}
    M &< \frac{(r-1)n}{(s-2)(r-1)+1}
    \sum_{j=0}^{s-3}\|\boldsymbol{v}(x_j)\| + \frac{(s-2)n}{(s-2)(r-1)+1}\\
    &= \frac{n}{(s-2)(r-1)+1}
    \big(
    (s-3)(r-1)+1
    \big)(s+r-3),\notag
\end{align}
which is a contradiction to the assumption (\ref{ineq:contrary}).
This completes the proof of~\cref{thm:transform to vectors}.

\subsection{Proof of~\cref{prop:exist Ks-2 copy}}

    We begin with the vertex \( x_0 \in V(H) \) with maximum weight, that is,  
\[
\|\boldsymbol{v}(x_0)\| = \max\{\|\boldsymbol{v}(y)\| : y \in V(H)\} := t - g,
\]  
where \( g \geq 0 \).  
By \eqref{ineq:contrary},  \( g \leq \frac{(r-1)(s+r-3)}{(s-2)(r-1)+1} \). Set $a_0=1$. We shall iteratively find the other $s-3$ vertices $x_1,\ldots,x_{s-3}$ such that the following holds.
    \begin{itemize}
    \item $\|\boldsymbol{v}(x_{i})\| = a_{i}(t-g) + g$ for $i\in [s-3]$, where $a_{1}\ge a_{2}\ge\cdots\ge a_{s-3}$ and \begin{equation}\label{ineq:recursive}
    a_{i}
    \ge
    \frac{(s-3-\sum_{\ell=0}^{i-1}a_{\ell})(r-1)+1}{(s-2-i)(r-1)+1}
    \end{equation}
\end{itemize} 

Assume that we have already found the vertices $x_0,x_{1},\ldots,x_{j-1}$ for $1\le j\le s-3$ satisfying~\eqref{ineq:recursive} and for each $h\in [j-1]$, $x_{h}\in \bigcap_{0\le h'<h} N_{G}(x_{h'})$. 
Suppose that for any vertex $y\in \bigcap_{0\le h<j}N_{G}(x_{h})$, 
$$\|\boldsymbol{v}(y)\|\le b:= g+\frac{(s-3-\sum_{h=0}^{j-1}a_{h})(r-1)+1}{(s-2-j)(r-1)+1}\cdot (t-g).$$
Then as in~\eqref{eq:bound-M-above} and~\eqref{eq:bound-M-above2}, we have 
\begin{align*}
   M &\le \, (n - |N_{G}(x_{0})|)\|\boldsymbol{v}(x_0)\| + |\bigcap_{h=0}^{j-2}N_{G}(x_{h})\setminus N_{G}(x_{j-1})| \|\boldsymbol{v}(x_{j-1})\|  +|\bigcap_{h=0}^{j-1}N_{G}(x_{h})|\cdot b\\
 &\le  \frac{(r-1)n}{(s-2)(r-1)+1}
    \sum_{h=0}^{j-1}\|\boldsymbol{v}(x_h)\| + (j\cdot\delta(G)-(j-1)n)\cdot b\\
 &\le   \frac{(r-1)n}{(s-2)(r-1)+1} \cdot \Big(  \sum_{h=0}^{j-1} a_h (t - g) + (j-1)g \Big) + \frac{(s-2-j)(r-1) + 1}{(s-2)(r-1) + 1} \cdot bn . 
\end{align*}
This, together with~\eqref{ineq:contrary}, infers that
\begin{align*}
    b > t - \frac{
    (r-1)(t-g)(\sum_{h=0}^{j-1}a_{h} -j +1)
    }{(s-2-j)(r-1) +1}
    = g + \frac{
    (s-3-\sum_{h=0}^{j-1}a_{h})(r-1)+1
    }{(s-2-j)(r-1)+1} (t-g) \ge b,
\end{align*}
which is a contradiction.
Therefore, we can find a vertex $x_j\in \bigcap_{0\le h<j} N_{G}(x_{h})$ with $\|\boldsymbol{v}(x_{j})\| = a_j(t-g) + g\ge b$ such that $a_j\le a_{j-1}$ and
\begin{align*}
    a_j\ge  \frac{
    (s-3-\sum_{h=0}^{j-1}a_{h})(r-1)+1
    }{(s-2-j)(r-1)+1}.
\end{align*}

We have thus found $x_0,\ x_1,\ldots,\ x_{s-3}$ which induce a copy of $K_{s-2}$ and satisfies~\eqref{eq:max weight} by construction. We are left to verify~\eqref{eq:sum=1+t(s-3)}. Let $A_j := \sum_{0\le h \le j}a_h$ for each $j\in\{ 0,1,\ldots,s-3\}$.
By (\ref{ineq:recursive}) we have
\begin{align*}
    (A_j - A_{j-1}) \ge 
    \frac{ (r-1)(s-3-A_{j-1}) +1}
    { (s-2-j)(r-1) +1},
\end{align*}
which implies
$$\beta_jA_j-\beta_{j+1}A_{j-1}\ge \beta_1.$$
Solving this recursion, we get
$$A_j\ge \frac{\beta_{j+1}}{\beta_1}A_0+\beta_{j+1}\beta_1\Big(\frac{1}{\beta_{j+1}\beta_{j}}+\frac{1}{\beta_{j}\beta_{j-1}}+\cdots+\frac{1}{\beta_2\beta_1}\Big).$$
As $\frac{1}{\beta_{j+1}}-\frac{1}{\beta_{j}}=\frac{r-1}{\beta_{j+1}\beta_{j}}$ and $A_0 =a_0= 1$, we see that $A_j\ge \frac{\beta_{j+1}}{\beta_1}+\frac{\beta_1-\beta_{j+1}}{r-1}$. In particular, since $\beta_{s-2}=0$, 
\begin{equation}\label{eq:As-3}
 A_{s-3}\ge \frac{\beta_1}{r-1}=\frac{(s-3)(r-1)+1}{r-1}>s-3.    
\end{equation}   

Now, on one hand, as $x_0,\ldots,x_{s-3}$ induce a $K_{s-2}$, by~\eqref{eq:onlyone}, 
$$\sum_{0\le h\le s-3} \|\boldsymbol{v}(x_h)\| \le 1\cdot (s-2) + (t-1)\cdot (s-3) = t(s-3)+1.$$
On the other hand, by~\eqref{eq:As-3},
    $$\sum_{0\le h\le s-3} \|\boldsymbol{v}(x_h)\|\ge
    A_{s-3}(t-g) + (s-3)g > t(s-3).$$
Thus, $\sum_{0\le h\le s-3} \|\boldsymbol{v}(x_h)\|= t(s-3) +1$ as desired. This completes the proof of~\cref{prop:exist Ks-2 copy}.


\section{Dense graphs without small odd cycles}\label{sec:OddCycles}
We will take advantage of the following result in~\cite[Proposition~3.5]{2020COMBHomoOddCycle} and we encourage the interested readers to the short proof in~\cite{2020COMBHomoOddCycle}.
\begin{lemma}[\cite{2020COMBHomoOddCycle}]\label{lemma:NoSmallOddCycles}
   For fixed integer $k\ge 2$, let $\varepsilon>0$ and $G$ be an $n$-vertex $C_{2k-1}$-free graph with minimum degree $\delta(G)\ge (\frac{1}{2k-1}+\varepsilon)n$, then $G$ is $C_{\ell}$-free for every odd $\ell$ with $k\le\ell\le 2k-1$.
\end{lemma}

\subsection{Warm up: $C_{5}$-free graphs}\label{sec:warmup}
To better illustrate our framework, we begin by presenting a proof specifically for $C_{5}$-free graphs. Compared to the previous proof in~\cite{2019JGTC3C5}, our proof is simpler, and the constant dependence is better than $2^{2^{2^{\textup{Poly}(\frac{1}{\varepsilon})}}}$ in~\cref{thm:SharpThreshold}. We restate the statement here for clarity. 

\begin{theorem}\label{thm:C5}
For any $\varepsilon>0$, let $G$ be an $n$-vertex maximal $C_{5}$-free graph with minimum degree $\delta(G)\ge (\frac{1}{5}+\varepsilon)n$, then $G=H[\cdot]$ for some $H$, where $|H|\le 2^{(\frac{1000}{\varepsilon})^{2}}$.
\end{theorem}

By~\cref{lemma:Bounded VC-dimension}, the VC-dimension of an $n$-vertex maximal $C_{5}$-free graph with minimum degree $(\frac{1}{5}+\varepsilon)n$ is a constant. In fact, it is at most $2$ by the following result.
\begin{lemma}[\cite{2019JGTC3C5}]\label{lemma:InducedC6Free}
    Let $G$ be an $n$-vertex maximal $C_{5}$-free graph with $\delta(G)>\frac{n}{5}$. Then $G$ does not contain an induced $6$-cycle.
\end{lemma}

\begin{prop}\label{claim: VCOnlyTwo}
   Let $G$ be an $n$-vertex maximal $C_{5}$-free graph with $\delta(G)>\frac{n}{5}$. Then $\textup{VC}(G)\le 2$.
\end{prop}
\begin{proof}
Suppose that there exists a shattered set of size $3$, say $X:=\{x_{1},x_{2},x_{3}\}$. Let $y_{1},y_{2},y_{3}$ be the vertices satisfying that $y_{1}\in (N(x_{1})\cap N(x_{2}))\setminus N(x_{3})$, $y_{2}\in (N(x_{2})\cap N(x_{3}))\setminus N(x_{1})$ and $y_{3}\in (N(x_{3})\cap N(x_{1}))\setminus N(x_{2})$. By~\cref{lemma:NoSmallOddCycles}, $G$ is triangle-free, thus $\{x_1,x_2,x_3\}$ and $\{y_1,y_2,y_3\}$ are both independent set. Therefore $x_{1}y_{1}x_{2}y_{2}x_{3}y_{3}$ forms an induced $6$-cycle, contradicting~\cref{lemma:InducedC6Free}.
\end{proof}

\begin{proof}[Proof of~\cref{thm:C5}]
Take $s=\frac{\varepsilon n}{50}$, by~\cref{lemma:Partition}, $V(G)$ can be partitioned into $V_{1}\sqcup\cdots\sqcup V_{r}$ with $r\le 3e(\frac{2e n}{s})^{2}\le (\frac{999}{\varepsilon})^{2}$, such that any pair of vertices in the same part has at least $(\frac{1}{5}+\varepsilon)n-2s$ common neighbors. Moreover by~\cref{lemma:NoSmallOddCycles}, $G$ is triangle-free, therefore $V_{i}$ is an independent set for each $i\in [r]$. We then refine the above partition as follows. For each $i\in [r]$, we partition $V_{i}$ into at most $m=2^{r}$ parts, say $V_{i}:=V_{i}^{1}\sqcup\cdots\sqcup V_{i}^{m}$ such that for any $j\in [m]$ and $\ell\in [r]$, any pair of vertices $x,y\in V_{i}^{j}$ satisfies $N(x)\cap V_{\ell}=\emptyset$ if and only if $N(y)\cap V_{\ell}=\emptyset$.

It suffices to show the following claim.
\begin{claim}\label{claim:C5CompleteOrEmpty}
    For distinct $i_{1},i_{2}\in [r]$ and $j_{1},j_{2}\in [m]$, $V_{i_{1}}^{j_{1}}$ and $V_{i_{2}}^{j_{2}}$ are either complete, or anti-complete.
\end{claim}
\begin{poc}
    Suppose that there are vertices $u_{1},u_{2}\in V_{i_{1}}^{j_{1}}$ and $v\in V_{i_{2}}^{j_{2}}$ such that $u_{1}v\in E(G)$ and $u_{2}v\notin E(G)$. Since $G$ is maximal $C_{5}$-free, there is a path of length $4$ between $u_{2}$ and $v$, say $u_{2}w_{1}w_{2}w_{3}v$, where $w_{p}\in V_{a_{p}}^{b_{p}}$ for $p\in [3]$, $a_{p}\in [r]$ and $b_{p}\in [m]$.
    
    As $\delta(G)\ge(\frac{1}{5}+\varepsilon)n$, there is a pair of vertices in $\{v,u_{2},w_{1},w_{2},w_{3}\}$ having at least $\frac{\varepsilon n}{10}$ common neighbors. This pair cannot be adjacent since $G$ is triangle-free and hence it must be one of non-adjacent pairs $(v,u_2),(v,w_{1}),(v,w_{2}),(u_{2},w_{3}),(u_{2},w_{2})$, $(w_{1},w_{3})$. We shall derive a contradiction in each case. Note first that as each $V_i$ is an independent set, $a_{1}\neq a_{2}$, $a_{2}\neq a_{3}$, $a_{3}\neq i_{2}$, $i_{1}\neq i_{2}$ and $i_{1}\neq a_{1}$.
    \begin{itemize}
      \item $(v,u_2)$: Since $u_1,u_2\in V_{i_1}$, $|N(u_1)\cap N(v)\cap N(u_{2})|\ge \frac{\varepsilon n}{10}-2s>0$. One such common neighbor forms a triangle with $u_1$ and $v$, a contradiction.
        \item $(v,w_{1})$: Then we can pick a vertex $h_{1}\in N(v)\cap N(w_{1})$ so that $vh_{1}w_{1}w_{2}w_{3}$ is a copy of $C_{5}$, a contradiction. We have a similar contradiction for $(u_2,w_3)$.
        \item $(v,w_{2})$: Since $u_{1}\in V_{i_{1}}^{j_{1}}$ is adjacent to $v\in V_{i_{2}}^{j_{2}}$ and $i_1\neq i_2$, by the construction of the partition $u_{2}$ has a neighbor $z_{1}\in V_{i_{2}}$. 
        \begin{itemize}
            \item If $z_1\neq w_1,w_2$, then as $|N(z_{1})\cap N(v)\cap N(w_{2})|\ge \frac{\varepsilon n}{10}-2s>0$, we can pick a vertex $h_{2}\in N(z_{1})\cap N(v)\cap N(w_{2})$ so that $u_{2}z_{1}h_{2}w_{2}w_{1}$ is a copy of $C_{5}$, a contradiction.
            \item If $z_1=w_2$, then $u_2w_1w_2$ is a triangle, a contradiction.
            \item If $z_1=w_1$, then $v,w_1\in V_{a_1}$ and they have at least $(\frac{1}{5}+\varepsilon)n-2s$ common neighbors, from which we can pick one, say $x$ such that $xw_1w_2w_3v$ in a copy of $C_5$, a contradiction.
        \end{itemize}
        \item $(u_{2},w_{2})$: Since $u_{1},u_{2}\in V_{i_{1}}$, $|N(u_{1})\cap N(u_{2})\cap N(w_{2})|\ge \frac{\varepsilon n}{10}-2s>0$, then we can pick a vertex $h_{4}\in N(u_{1})\cap N(u_{2})\cap N(w_{2})$ so that $vu_{1}h_{4}w_{2}w_{3}$ is a copy of $C_{5}$, a contradiction.
        \item $(w_{1},w_{3})$: Since $u_{2}\in V_{i_1}$ is adjacent to $w_{1}\in V_{a_1}$ and $i_1\neq a_1$, $u_{1}$ has a neighbor $z_{2}\in V_{a_{1}}$. Note that $z_2\neq w_1$ for otherwise $u_1w_1w_2w_3v$ is a $C_5$. 
        \begin{itemize}
            \item If $z_2\neq v,w_3$, then as $|N(z_{2})\cap N(w_{1})\cap N(w_{3})|\ge\frac{\varepsilon n}{10}-2s>0$,  there is a vertex $h_{5}\in N(z_{2})\cap N(w_{1})\cap N(w_{3})$ so that $w_{3}vu_{1}z_{2}h_{5}$ is a copy of $C_{5}$, a contradiction.
            \item If $z_2=w_3$, then $u_1w_3v$ is a triangle, a contradiction.
            \item If $z_2=v$, then $v,w_1\in V_{a_1}$. As above we can pick one of their common neighbor $x$ such that $xw_1w_2w_3v$ in a copy of $C_5$, a contradiction.
        \end{itemize}
    \end{itemize}

This completes the proof of the claim.
\end{poc}
By~\cref{claim:C5CompleteOrEmpty}, $G$ is a blowup of a graph on at most $r\cdot 2^{r}\le 2^{(\frac{1000}{\varepsilon})^{2}}$ vertices.
\end{proof}

\subsection{Dense graphs without short odd cycles have bounded VC-dimension}
In this part, we will prove~\cref{lemma:Bounded VC-dimension}, which states that the VC-dimension of any maximal $C_{2k-1}$-free graph with large minimum degree can be bounded by a constant only depending on $k$. We will need the well-known Bollob\'{a}s set-pair inequality~\cite{1965BollobasSetpair}. 

\begin{theorem}[\cite{1965BollobasSetpair}]\label{thm: bollobas}
Suppose that $A_1, A_2,\ldots, A_m$ are sets of size $k$ and $B_1, B_2,\ldots, B_m$ are sets of size $\ell$ satisfying 
$A_i\cap B_i = \emptyset$ for each $i\in [m]$ and $A_i\cap B_j \neq \emptyset$ for each $i\neq j$.
Then $m\le \binom{k+\ell}{k}$.
\end{theorem}

\begin{proof}[Proof of Lemma~\ref{lemma:Bounded VC-dimension}]
    Suppose that there is an $n$-vertex maximal $C_{2k-1}$-free graph $G$ with $\delta(G)\ge (\frac{1}{2k-1}+\eps)n$ and VC-dimension at least $C_0$, where $C_0>(2k-1)^3\binom{2k-2}{k-1} 2^{(2k-1)^{2k-2}}$.
    Let $S_0=\{s_1,\ldots, s_{C_0} \}\subseteq V(G)$ be a shattered set. Then there is a set of distinct vertices $T_0=\{v_1,\ldots,v_{C_0}\}$ such that for each $i\in [C_0]$, $N(v_i)\cap S = S\setminus \{s_i\}$. As $G$ is $C_{2k-1}$-free and hence also $K_{2k-1}$-free, we have $|S_0\cap T_0|\le 2k-2$.
    Let $C=\big((2k-1)^2\binom{2k-2}{k-1} +1\big)2^{(2k-1)^{2k-2}}$. As $C+2k-2< C_0$, by relabeling, we may assume that $S=\{s_1,\ldots, s_{C} \}$ is disjoint from $T=\{v_1,\ldots, v_{C} \}$.

    Since $G$ is maximal $C_{2k-1}$-free, for each pair of non-adjacent vertices, there exists a path of length $2k-2$ connecting them. Thus, for each $i\in [C]$, there exists a path $P_i = p_i^0 p_i^1 \cdots p_i^{2k-2}$ of length $2k-2$ connecting $s_i$ and $v_i$ in $G$, where $s_i = p_i^0$ and $v_i = p_i^{2k-2}$. We first upper bound the number of intersecting pairs of paths via~\cref{thm: bollobas}.

    \begin{claim}\label{claim: small intersecting paths}
        For each $i\in [C]$, we have $\big|\{j\in [C]: V(P_j)\cap V(P_i)\neq \emptyset \}\big|
        \le
        (2k-1)^2 \binom{2k-2}{k-1}.$
    \end{claim}

    \begin{poc}
        By symmetry, it suffices to prove the case of $i=C$. Suppose that the statement is false. For each $j\in [C-1]$, if $V(P_C)\cap V(P_j)\neq \emptyset$, let $\ell_{j}$ be the largest index in $[0,2k-2]$ such that $p_j^{\ell_{j}}\in V(P_C)$.
        By pigeonhole principle, we can find at least $L = \bigg\lceil\frac{(2k-1)^2 \binom{2k-2}{k-1}+1}{(2k-1)^2}\bigg\rceil = \binom{2k-2}{k-1}+1$ indices $1\le k_1< \cdots < k_L\le C-1$, such that $\ell_{k_1}  = \cdots = \ell_{k_L}=:\ell$ and $p_{k_1}^{\ell_{k_1}}  = \cdots
        = p_{k_L}^{\ell_{k_L}}$. For simplicity, we assume that $k_j = j$ for each $j\in [L]$. For each $j\in [L]$, we define $V_j^1$ and $V_j^2$ to be $V_j^1 = \{p_j^0, p_j^1,  \ldots, p_j^{\ell-1}\}$ and $V_j^2 = \{p_j^{\ell+1}, p_j^{\ell+2},  \ldots, p_j^{2k-2}\}$ respectively. Observe that, for each $j\in [L]$, we have $|V_j^1| = \ell$, $|V_j^2| = 2k-2-\ell$ and $V_j^1\cap V_j^2 = \emptyset$.
        
        Now we consider these $L$ set pairs $\{(V_j^1, V_j^2)\}_{j\in [L]}$,
        since $L> \binom{2k-2}{k-1} \ge \binom{2k-2}{\ell}$, by~\cref{thm: bollobas}, there exist distinct indices $j_1$ and $j_2$ in $[L]$ such that $V_{j_1}^1 \cap V_{j_2}^2 = \emptyset$.
        Without loss of generality, we may assume $j_1 = 1$ and $j_2 = 2$.
        Then we can find a cycle $p_1^0 p_1^1\cdots p_1^{\ell} p_2^{\ell+1} p_2^{\ell+2}\cdots p_2^{2k-2} p_1^0$ of length $2k-1$ since $p_1^0=s_1$ is adjacent to $p_2^{2k-2}=v_2$, a contradiction.
    \end{poc}

    By \Cref{claim: small intersecting paths}, we can greedily find a subset, say without loss of generality $[C']\subseteq [C]$, where $C' = \frac{C}{(2k-1)^2 \binom{2k-2}{k-1}+1}=2^{(2k-1)^{2k-2}}$, such that for any distinct indices $i,j\in [C']$, we have $V(P_i)\cap V(P_j) = \emptyset$. We will select a subcollection of these paths using the following claim.
    
    \begin{claim}\label{claim:large set of common neighbors}
        For any subset $U\subseteq V(G)$ with $|U|=K > (2k-1)^{2k-1}$, there is a subset $U'\subseteq U$ with $|U'|\ge K^{\frac{1}{2k-1}}$ such that any pair of vertices in $U'$ have at least $5k$ common neighbors.
    \end{claim}

    \begin{poc}
        We build an auxiliary graph $H$ with vertex set $U$ such that two vertices in $H$ are adjacent if they share at least $5k$ common neighbors in $G$.
        Since $n$ is sufficiently large, we have
        \begin{equation*}
            (2k-1)\bigg(\frac{1}{2k-1}+\eps\bigg)n - 5k \binom{2k-1}{2} > n,
        \end{equation*}
        which implies that, for any $2k-1$ distinct vertices in $G$, there are two of them sharing at least $5k$ common neighbors.
        Therefore, the independence number $\alpha(H)<2k-1$. Then the upper bound for Ramsey number $R(s,t)\le \binom{s+t-2}{s-1}$~\cite{1935ES} yields that
         $H$ contains a clique of size at least $K^{\frac{1}{2k-1}}$, which can be chosen as the subset $U'$ in the claim. This finishes the proof.
    \end{poc}

    As $C{'}^{\frac{1}{(2k-1)^{2k-2}}} \ge 2$, we can repeatedly apply~\cref{claim:large set of common neighbors} to obtain two paths among $\{P_i:i\in[C']\}$, say $P_1,P_2$, such that $p_1^\ell$ and $p_2^\ell$ share at least $5k$ common neighbors in $G$ for each $\ell\in \{1,\ldots,2k-2\}$. Note also that $p_1^0=s_1$ and $p_2^0=s_2$ are in $S$, and thus they share at least $2^{C-2} \ge 5k$ common neighbors. Denote $V_0 = V(P_1)\cup V(P_2)$. Note that $G[V_0]$ contains a spanning cycle of length $4k-2$; we denote this cycle by $C^*=u_1u_2\cdots u_{4k-2}u_1$, where $u_i = p_1^{i-1}$ for any $i\in [2k-1]$ and $u_{j+2k-1} = p_2^{j-1}$ for each $j\in [2k-1]$.
    By the argument above, we know that for each $i\in [2k-1]$, $u_i$ and $u_{i+2k-1}$ share at least $5k$ common neighbors. 

    \begin{claim}\label{claim:at most two neighbors}
        There exists a vertex $y\in V(G)\setminus V_0$ such that $|N(y)\cap V_0|\ge 3$.
        \end{claim}
\begin{poc}
    Suppose that for every vertex $y\in V(G)\setminus V_0$, $|N(y)\cap V_0|\le 2$.
    We count $e(V_0,V\setminus V_0)$, the number of edges in $G$ between $V_0$ and $V(G)\setminus V_0$. On one hand, by assumption,
    $$ e(V_0, V(G)\setminus V_0) \le 
    \sum_{y\in V\setminus V_0}|N(y)\cap V_0| \le 2|V(G)\setminus V_0| \le 2n. $$  
  On the other hand, by the minimum degree condition,
\begin{equation*}
    e(V_0, V(G)\setminus V_0) \ge \sum_{j = 1}^{4k-2} \big(|N(u_{j})|-|V_0|\big)
    \ge \sum_{j = 1}^{4k-2}\bigg( \bigg(\frac{1}{2k-1}+\eps\bigg)n -(4k-2) \bigg)
    > 2n,   
\end{equation*}
a contradiction.
\end{poc}  

Let $y\in V(G)\setminus V_0$ be a vertex with $|N(y)\cap V_0|\ge 3$ given by the above claim. We shall make use of the cycle $C^*=u_1u_2\cdots u_{4k-2}u_1$ and edges from $y$ to $V_0$ to obtain a copy of $C_{2k-1}$ or $C_{2k-3}$, which would contradict~\cref{lemma:NoSmallOddCycles}. Recall that every long diagonal of $C^*$, i.e.,~$(u_i,u_{i+2k-1})$, $i\in[2k-1]$, share at least $5k$ common neighbors. Denote the subset consisting of vertices with odd indices in $V_0$ by $V_0^{odd}$, and the remaining subset by $V_0^{even}$. The indices are taken modulo $4k-2$.

    \begin{claim}\label{claim:only one even}
        For any $i\in[2k-1]$, if $u_{2i-1}\in N(y)$, then $N(y)\cap V_0^{even} \subseteq \{ u_{2k+2i-2} \}$.
    \end{claim}

    \begin{poc}
    By the symmetry of $C^*$, we may assume $i=1$ and $u_1\in N(y)$. Suppose that $u_{2p}\in N(y)$ for some $p\neq k$, again by symmetry of $C^*$, we may assume that $p\in [k-1]$. 
    
    If $p\in \{k-2,k-1\}$, then $yu_1u_2\cdots u_{2p}y$ forms a copy of cycle of length $2p+1\in \{2k-3, 2k-1\}$, a contradiction.
        
        If instead $1\le p\le k-3$, recall that both $\{u_1, u_{2k}\}$ and $\{u_{k+p-3}, u_{3k+p-4}\}$  share at least $5k$ common neighbors. We can pick distinct vertices $x\in N(u_1)\cap N(u_{2k})$ and $x'\in N(u_{k+p-3})\cap N(u_{3k+p-4})$ such that $\{x, x'\}\cap (\{ y \}\cup V_0) = \emptyset$.
        Then $y u_1 x u_{2k} u_{2k+1} \cdots u_{3k+p-4} x' u_{k+p-3} u_{k+p-4} \cdots u_{2p} y$ forms a copy of cycle of length $2k-1$, a contradiction. 
    \end{poc}

    \begin{claim}\label{claim:only two odd}
    $N(y)\cap V_0^{odd}\subseteq \{u_{2i-1},u_{2i+1}\}$ for some $i\in[2k-1]$.
    \end{claim}
    \begin{poc}
    We may assume $N(y)\cap V_0^{odd}\neq\varnothing$ and by the symmetry of $C^*$, we may further assume $u_1\in N(y)$. It suffices to prove that  $Z_1=N(y)\cap V_0^{odd}\setminus\{u_{1}\}\subsetneq\{ u_3, u_{4k-3} \}$. 
    
    Suppose $N(y)$ contains $u_{2p+1}$ for $p\neq 1,2k-2$. Again by symmetry of $C^*$, we may assume $2\le p\le k-1$. Since $u_{p-1}$ and $u_{p+2k-2}$ share at least $5k$ common neighbors, we can take a common neighbor $x$ of $u_{p-1}$ and $u_{p+2k-2}$ such that $x\not\in \{y\}\cup V_0$. Then $ y u_1u_2 \cdots u_{p-1} x u_{p+2k-2} u_{p+2k-3} \cdots u_{2p+1}y$ forms a cycle of length $2k-1$, a contradiction to~\cref{lemma:NoSmallOddCycles}. 
    
    Thus $Z_1\subseteq \{u_3,u_{4k-3}\}$. Now suppose without loss of generality that $u_3\in Z_1\subseteq N(y)$, by the above argument, we must have $N(y)\cap V_0^{odd}\setminus\{u_{3}\}\subseteq \{u_1,u_5\}$. Consequently, as $k\ge 3$, we have $N(y)\cap V_0^{odd}=\{u_1,u_3\}$ as claimed. 
    \end{poc}

   By symmetry of $C^*$, we may assume without loss of generality that $N(y)\cap V_0^{odd}\neq\varnothing$ and $u_1\in N(y)$. Now we must have $(N(y)\cap V_0^{odd})\setminus \{u_1\}\neq\varnothing$ for otherwise $|N(y)\cap V_0|\le 2$ by~\cref{claim:only one even}, contradicting~\cref{claim:at most two neighbors}. Then by~\cref{claim:only two odd}, we may assume $N(y)\cap V_0^{odd}=\{u_1,u_3\}$ and therefore by~\cref{claim:only one even}, $N(y)\cap V_0^{even}=\varnothing$, contradicting~\cref{claim:at most two neighbors} again. 
   
   This completes the proof of~\cref{lemma:Bounded VC-dimension}. 
\end{proof}

\subsection{Main lemma and its consequences}
We first describe how to partition a dense graph with no short odd cycles based on~\cref{lemma:Bounded VC-dimension} and~\cref{lemma:Partition}, along with some auxiliary results. A \emph{walk} of length $k$ is a sequence of vertices $v_{1}v_{2}\ldots v_{k+1}$ with $v_{i}v_{i+1}\in E(G)$ for $1\le i\le k$. A walk is said to be \emph{closed} if the first and last vertices are the same. 
The following simple fact about closed walks will be useful in our proof.

\begin{fact}\label{claim:odd walk to odd cycle}     
    If a closed walk $W$ of odd length is not an odd cycle, then it contains an odd cycle $C$ such that at least one vertex in $C$ appears at least twice in $W$.\footnote{We consider the first and last vertex of a closed walk appear only once in the walk.} Moreover $3\leq |C|\leq |W|-2$.
\end{fact}
We next describe our partitions in details.
\begin{definition}\label{def:PartitionRules}
   Let $G$ be an $n$-vertex graph with VC-dimension at most $d$. Let $0<\gamma<1$ be a real number and $i\ge 0$ be an integer, we define the \emph{$\gamma$-difference $i$-th partition of $G$}, denoted by $\mathscr{P}_i(G,\gamma)$, as follows.

    \begin{itemize}
        \item By~\cref{lemma:Partition}, $V(G)$ can be partitioned into $V_{1}\sqcup V_{2}\sqcup\cdots\sqcup V_{r}$ with $r\le e(d+1)(\frac{4e}{\gamma})^{d}$ such that for any $u,v$ in the same part, $|N(u)\triangle N(v)|\le \gamma n$. We denote this partition by $\mathscr{P}_0(G,\gamma)$.

        \item For $i\ge0$, we refine $\mathscr{P}_i(G,\gamma)$ to derive $\mathscr{P}_{i+1}(G,\gamma)$ as follows. Partition each  class of $\mathscr{P}_i(G,\gamma)$ into at most $2^{|\mathscr{P}_i(G,\gamma)|}$ new classes such that  for any $u$ and $v$ in the same class of $\mathscr{P}_i(G,\gamma)$, they are in the same class of $\mathscr{P}_{i+1}(G,\gamma)$ if and only if for any class $P$ of $\mathscr{P}_i(G,\gamma)$, we have
$$N(u)\cap P=\emptyset\quad \Leftrightarrow \quad  N(v)\cap P=\emptyset.$$
\end{itemize}

\begin{definition}\label{defn:aux-H}
For each $i\ge 0$, we define an auxiliary graph $H_i(G,\gamma)$, in which each vertex corresponds to a class of $\mathscr{P}_{i}(G,\gamma)$. Two vertices $X$ and $Y$ are adjacent in $H_i(G,\gamma)$ if and only if $(X,Y)$ induces an edge in $G$.
\end{definition}

\end{definition}
For brevity, we sometimes omit $G$ and $\gamma$, simplifying it to $\mathscr{P}_i$, when it is clear from the content. Let $i\le j$ be integers, note that $\mathscr{P}_{j}(G,\gamma)$ is a refinement of $\mathscr{P}_{i}(G,\gamma)$. For a class $P$ of $\mathscr{P}_{j}$, let $P^{(i)}$ be the unique class of $\mathscr{P}_{i}$ that contains $P$. Furthermore, for a vertex $v\in V(G)$, let $v^{(i)}$ be the unique class of $\mathscr{P}_{i}$ that contains $v$. For $i\ge 0$, we say a class of $\mathscr{P}_{i}$ is \emph{singular} if it contains exactly one vertex.

\begin{definition}[Singular cycles]
    For $i\ge 1$, a cycle in $H_i(G,\gamma)$ (resp. $G$) is \emph{singular} if for any class $P\in V(H_{i}(G,\gamma))$ (resp. any vertex $v\in V(G)$) in this cycle, $P^{(1)}$ (resp. $v^{(1)}$) is singular, and \emph{non-singular} otherwise. 
\end{definition}

Based on the above definitions, we now introduce the following main lemma, and show how it  yields~\cref{thm:ByproductOne} and the upper bound of~\cref{thm:SharpThreshold}.

\begin{lemma}\label{lemma:main}
For any integer $k\ge 2$ and any $\varepsilon>0$, for an $n$-vertex graph $G$ with VC-dimension at most $d$.
If one of the followings holds, 
    \begin{enumerate}
        \item[\textup{(1)}] $G$ is $\{C_{2k-1},C_{2k-3}\}$-free,
\footnote{Here $\{C_{3},C_{1}\}$-free just means $\{C_{3}\}$-free.} and $\delta(G)\ge(\frac{1}{2k-1}+ 2k\varepsilon )n$;
        \item[\textup{(2)}] $G$ is $\{C_{2k+1},C_{2k-1}\}$-free, and $\delta(G)\ge \varepsilon n$,
    \end{enumerate} 
then $G$ is non-singular $C_{2\ell-1}$-free for $2\le\ell\le k-1$, and $G \xrightarrow{\textup{hom}} H_{k}(G,\frac{\varepsilon}{3})$, where $H_{k}(G,\frac{\varepsilon}{3})$, as defined in~\cref{defn:aux-H}, is $C_{2k-1}$-free and non-singular $C_{2\ell-1}$-free for any $2\leq \ell\le k-1$.
\end{lemma}

\begin{proof}[Proof of~\cref{thm:ByproductOne}]
Let $k,d\ge 2$, $\varepsilon>0$ and $G$ be an $n$-vertex graph with $\delta(G)\ge \varepsilon n$ and $\textup{VC}(G)\le d$. If $G$ is $\mathscr{C}_{2k+1}$-free, then by~\cref{lemma:main}(2), $G$ is homomorphic to $H:=H_k(G,\frac{\varepsilon}{3})$ where $H$ is $C_{2k-1}$-free and non-singular $C_{2\ell-1}$-free for $2\le\ell\le k-1$. Since $G$ is $\mathscr{C}_{2k+1}$-free and any singular $C_{2\ell-1}$ in $H$ corresponds to a copy of $C_{2\ell-1}$ in $G$, we can see $H$ is singular $C_{2\ell-1}$-free for $2\le\ell\le k-1$. Therefore $H$ is $\mathscr{C}_{2k-1}$-free. Thus, $\delta_{\textup{hom}}^{\textup{VC}}(\mathscr{C}_{2k+1};\mathscr{C}_{2k-1})=0$.

Note that $\delta_{\textup{hom}}^{\textup{VC}}(\{C_{2k+1},C_{2k-1}\};C_{2k-1})=0$ follows directly from~\cref{lemma:main}(2).
\end{proof}

Define the \emph{tower function} recursively as $\textup{tw}_{1}(x)=x$ and $\textup{tw}_{i+1}(x)=\textup{tw}_{i}(x)2^{\textup{tw}_{i}(x)}$. For the upper bound of~\cref{thm:SharpThreshold}, we shall prove that $|H|\le \textup{tw}_{k}(r)$, $r\le e(d+1)(\frac{12ke}{\varepsilon})^{d}$ and $d\le (2k-1)^3\binom{2k-2}{k-1} 2^{(2k-1)^{2k-2}}$. 

\begin{proof}[Proof of upper bound for~\cref{thm:SharpThreshold}]
Let $G$ be an $n$-vertex maximal $C_{2k-1}$-free graph with minimum degree at least $(\frac{1}{2k-1}+\varepsilon)n$. Then $\textup{VC}(G)$ is at most $d:=(2k-1)^3\binom{2k-2}{k-1} 2^{(2k-1)^{2k-2}}$ by~\cref{lemma:Bounded VC-dimension}, and $G$ is $C_{2k-3}$-free by~\cref{lemma:NoSmallOddCycles}. Note that $G$ satisfies the first condition in~\cref{lemma:main}, then $G \xrightarrow{\textup{hom}} H$, where $H=H_k(G,\frac{\varepsilon}{6k})$ is $C_{2k-1}$-free and does not contain non-singular $C_{2\ell-1}$ for any $2\le \ell\le k-1$. 

We shall prove that $G$ is isomorphic to $G^\prime =H[\cdot]$, where the size of each part is based on the partition $\mathscr{P}_k(G,\frac{\varepsilon}{6k})$. Since $G$ is maximal $C_{2k-1}$-free and $G\subseteq G'$, it suffices to show that $G^\prime$ is $C_{2k-1}$-free. Suppose that $G^\prime$ contains a copy of $C_{2k-1}$, say $v_1v_2\cdots v_{2k-1}v_1$, which corresponds to a closed walk $W=v_1^{(k)}v_2^{(k)}\cdots v_{2k-1}^{(k)}v_{1}^{(k)}$ of length $2k-1$ in $H$. Since $H$ is $C_{2k-1}$-free, $W$ is not a cycle in $H$, by~\cref{claim:odd walk to odd cycle}, there exists an odd cycle $C$ of length at most $2k-3$ such that one of the vertex, say $v_i^{(k)}$, in $C$ appears at least twice in $W$. But then  $v_i^{(1)}$ contains at least two elements in $\{v_1,v_2,\dots,v_{2k-1}\}$ and hence $C$ is a non-singular odd cycle of length at most $2k-3$ in $H$, a contradiction to~\cref{lemma:main}.
    
 By~\cref{def:PartitionRules}, when $\gamma=\frac{\varepsilon}{6k}$, the number of classes of $\mathscr{P}_{0}$ is at most $r= e(d+1)(\frac{12ke}{\varepsilon})^{d}$, then we have $|H|\le \textup{tw}_{k}(r)$. This finishes the proof.
\end{proof}

\subsection{Proof of~\cref{lemma:main}}

Let $G$ be a graph that satisfies one of the conditions in~\cref{lemma:main}. We consider partitions of $G$ as in~\cref{def:PartitionRules}. Let us first establish some useful claims.
\begin{claim}\label{claim:from hom path to real walk}
    For any integer $t\ge 1$, let $P_{t}P_{t-1}\cdots P_{0}$ be a walk in $H_{t+1}(G,\frac{\varepsilon}{3})$, where each $P_i$ is a class of $\mathscr{P}_{t+1}(G,\frac{\varepsilon}{3})$. Then for any vertex $x_{t}\in V(G)$ belonging to the class $P_t$, there exists a walk $x_{t}x_{t-1}\cdots x_{0}$ in $G$ such that $x_i$ belongs to the class $P_i^{(i+1)}$.
\end{claim}
\begin{poc}     
    As $P_{t}P_{t-1}$ is an edge in $H_{t+1}(G,\frac{\varepsilon}{3})$, $(P_{t},P_{t-1})$ induces an edge of $G$, which also yields that $(P_{t},P_{t-1}^{(t)})$ induces an edge of $G$. Furthermore by the definition of $\mathscr{P}_{t+1}(G,\frac{\varepsilon}{3})$, every vertex in class $P_{t}$ has at least one neighbor in $P_{t-1}^{(t)}$. As $x_{t}$ belongs to class $P_{t}$, we can arbitrarily pick one neighbor of $x_{t}$ in $P_{t-1}^{(t)}$ to be $x_{t-1}$. 

When $t=1$, by the above argument the claim holds. We prove the full claim by induction on $t$. Let $t\ge 2$ and assume the claim holds for $t-1$. 
    We set $Q_{i}:=P_{i}^{(t)}$ for $0\le i\le t-1$. Notice that $Q_{t-1}Q_{t-2}\cdots Q_{0}$ is a walk in $H_{t}(G,\frac{\varepsilon}{3})$ and recall that $x_{t-1}$ belongs to class $Q_{t-1}=P_{t-1}^{(t)}$. By induction hypothesis, there exists $x_{t-1}x_{t-2}\cdots x_{0}$ in $G$ such that $x_i$ belongs to the class $Q_i^{(i+1)}=P_i^{(i+1)}$. Then $x_{t}x_{t-1}\cdots x_{0}$ is a walk in $G$ as desired.
\end{poc}

Denote $c_{d,\varepsilon}=2^{r}$, where $r=e(d+1)(\frac{12e}{\varepsilon})^{d}$, then $\mathscr{P}_{0}(G,\frac{\varepsilon}{3})$ has at most $r$ classes and $\mathscr{P}_{1}(G,\frac{\varepsilon}{3})$ has at most $c_{d,\varepsilon}$ classes. We say a path $xv_{1}v_{2}\cdots v_{s}y$ \emph{avoids} a subset $T$ if $v_{i}\notin T$ for each $i\in [s]$. 
\begin{claim}
    Let $t\in\mathbb{N}$ and $T\subseteq V(G)$ with $|T|<\frac{2\varepsilon n}{3} -c_{d,\varepsilon}-2t$. Then for any distinct vertices $x,y$ with $|N(x)\cap N(y)|\geq\frac{2\varepsilon n}{3}$, there exists a path of length $2t$ between $x,y$ that avoids $T$ in $G$.
\end{claim}
\begin{poc}
   We prove by induction on $t$. Since $|T|< \frac{2\varepsilon n}{3} -c_{d,\varepsilon}-2t\leq |N(x)\cap N(y)|-2$, the base case $t=1$ is trivial. For $t\geq 2$, assume that this claim holds for $t-1$. By $|N(x)\cap N(y)\setminus (T\cup\{x,y\})|\geq \frac{2\varepsilon n}{3} -(\frac{2\varepsilon n}{3} -c_{d,\varepsilon}-2t)\geq c_{d,\varepsilon}+2t\geq c_{d,\varepsilon}+4$. By pigeonhole principle we can always find two vertices $x^\prime,y^\prime\in N(x)\cap N(y)\setminus(T\cup\{x,y\})$ such that $x^\prime$ and $y^\prime$ belong to the same class of $\mathscr{P}_{0}(G,\frac{\varepsilon}{3})$. By~\cref{def:PartitionRules}, we have $|N(x')\triangle N(y')|\le \frac{\varepsilon n}{3}$. Therefore by the minimum degree conditions, 
   \begin{equation}\label{eq:sameclass}
       |N(x')\cap N(y')|\ge\varepsilon n-\frac{\varepsilon n}{3}\ge\frac{2\varepsilon n}{3}.
   \end{equation}
   Thus, we can apply the induction hypothesis to $T^\prime=T\cup\{x,y\}$, $x^\prime$, $y^\prime$ and $t^\prime=t-1$, there exists a path $P^\prime$ of length $2t^\prime=2t-2$ between $x^\prime$ and $y^\prime$ that avoids $T^\prime$ in $G$. Thus $P=xx^\prime P^\prime y^\prime y$ is the required path of length $2t$ that avoids $T$ in $G$. This finishes the proof.
\end{poc}

Note that thanks to~\eqref{eq:sameclass}, we can apply the above claim to find a path of length $2k-2$ between any two vertices in the same class of $\mathscr{P}_{0}(G,\frac{\varepsilon}{3})$. Consequently, each class of $\mathscr{P}_{0}(G,\frac{\varepsilon}{3})$ is an independent set as $G$ is $C_{2k-1}$-free. Moreover, we can immediately obtain the following.

\begin{claim}\label{claim:exists path}
    Let $T\subseteq V(G)$ be a subset with at most $100k$ vertices. Then for any two vertices $x,y$ in the same class of $\mathscr{P}_{0}(G,\frac{\varepsilon}{3})$ and for any positive even integer $s$ less than $100k$, there exists a path of length $s$ between $x,y$ which avoids $T$ in $G$. 
\end{claim}

\begin{claim}\label{claim:GNonSingularCycles}
    $G$ is non-singular $C_{2\ell-1}$-free for $2\le\ell\le k$.
\end{claim}
\begin{poc}
    The statement for $\ell=k$ follows from $G$ being $C_{2k-1}$-free. For $2\leq \ell \leq k-1$, we prove the claim by contradiction. Suppose there exists a non-singular $C_{2\ell-1}$ in $G$ with vertices $x_0, x_1,\ldots, x_{2\ell-2}$. First we show that there exists $C_{t}$ in $G$ for any odd number $2\ell+3\leq t< 100k$. 
    
    If there exists $i\neq j$, $x_i^{(0)}=x_j^{(0)}$, then for any positive even number $s<100k$, by~\cref{claim:exists path} there exists a path of length $s$ between $x_j$ and $x_i$ that avoids $T=\{x_0,x_1,\dots,x_{2\ell-2}\}$ in $G$. Since $x_{i}$ and $x_{j}$ belong to a copy of $C_{2\ell-1}$, there exists a path between $x_{i}$ and $x_{j}$ of length $\ell'$, where $1\le \ell'\le 2\ell-3$ is an odd number, such that all of the internal vertices belong to $T$. We can then combine these two paths between $x_i$ and $x_j$ to obtain a copy of $C_{t}$ for all odd number $2\ell -1 \leq t< 100k$.

    We may then assume that all $x_i^{(0)}$ are pairwise distinct. Without loss of generality, assume that $x_0^{(1)}$ is not singular, then we can pick $y_0\in x_0^{(1)}$ other than $x_0$. By~\cref{def:PartitionRules}, $y_0^{(0)}=x_0^{(0)}$. Moreover, since $x_{i}^{(0)}$ are pairwise distinct, $y_0\notin \{x_0,x_1,\dots,x_{2\ell-2}\}$. Since $x_{0}x_{1}\in E(G)$ and $y_0\in x_0^{(1)}$, by~\cref{def:PartitionRules} there exists a vertex $y_{1}$ in $x_1^{(0)}$ such that $y_{0}y_{1}\in E(G)$. If $y_{1}=x_1$, then $y_0 x_1\cdots x_{2\ell-2} x_0 T_{0} y_0$ is a copy of $C_{t}$ in $G$, where $T_0$ is a path between $x_0,y_0$ of even length $t-2\ell+1$ that avoids $ \{x_1,\ldots, x_{2\ell-2}\} $ in $G$ guaranteed by~\cref{claim:exists path}. On the other hand, if $y_{1}\neq x_1$, notice that $y_1^{(0)}=x_1^{(0)}$, therefore $y_1\notin \{y_0,x_0,x_1,\dots,x_{2\ell-2}\}$. Then $x_1\cdots x_{2\ell-2} x_0 T_0 y_0 y_1 T_1 x_1$ is a copy of $C_{t}$ in $G$, where $T_0$ is a path between $x_0,y_0$ of even length $t-2\ell-1$ that avoids $ \{y_1,x_1,\ldots,x_{2\ell-2}\}$ in $G$ and $T_1$ is a path between $x_1,y_1$ of length $2$ that avoids $ \{x_0,y_0,y_1,x_1,\dots,x_{2\ell-2} \}\cup V(T_{0})$ in $G$, both of which are guaranteed by~\cref{claim:exists path}.

    Thus, there exists a copy of $C_{t}$ in $G$ for all odd number $2\ell+3\leq t< 100k$. In the case of $\ell\leq k-2$, we can see $2\ell+3\leq 2k-1< 100k$, thus there exists a copy of $C_{2k-1}$ in $G$, a contradiction.
    Otherwise when $\ell =k-1$, then there exists a copy of $C_{2k-3}$ in $G$ by our assumption. By $2\ell+3\leq 2k+1< 100k$, there is a copy of $C_{2k+1}$ in $G$. Thus $G$ is neither $C_{2k-3}$-free nor $C_{2k+1}$-free, contradicting both conditions in~\cref{lemma:main}. This finishes the proof.
\end{poc}

By the mapping $v\rightarrow v^{(k)}$ and recalling that $v^{(k)}\subseteq v^{(0)}$ is an independent set, $G$ is homomorphic to $H_{k}(G,\frac{\varepsilon}{3})$. Note that a singular $C_{2k-1}$ in $H_{k}(G,\frac{\varepsilon}{3})$ corresponds to a copy of $C_{2k-1}$ in $G$, therefore $H_{k}(G,\frac{\varepsilon}{3})$ is singular $C_{2k-1}$-free since $G$ is $C_{2k-1}$-free. It remains to show that $H_{k}(G,\frac{\varepsilon}{3})$ does not contain non-singular $C_{2\ell-1}$ for any $2\leq \ell\leq k$. 

Suppose to the contrary that there exists a non-singular $C_{2\ell-1}$ in $H:=H_{k}(G,\frac{\varepsilon}{3})$ for some $2\leq \ell\leq k$. We denote this non-singular $C_{2\ell-1}$ in $H$ as $P_{-\ell+1}P_{-\ell+2}\cdots P_{-1}P_{0}P_{1}\cdots P_{\ell-2}P_{\ell-1}$. If $\ell=k$ and $G$ satisfies the first condition of~\cref{lemma:main}, then let us do the following preprocessing step that will be useful later. For each $i\in\{-k+1,\ldots,-1,0,1,k-1\}$, we pick an arbitrary vertex $x_i$ in the class $P_i$. As $\delta(G)\ge(\frac{1}{2k-1}+ 2k\varepsilon )n$, 
$$\sum_{i=-k+1}^{k-1} |N(x_i)|-n \ge (2k-1)\Big(\frac{1}{2k-1}+ 2k\varepsilon \Big)n-n=  \binom{2k}{2}\cdot 2\varepsilon n.$$
Thus, by the pigeonhole principle there exist two elements in $\{x_{-k+1},x_{-k+2},\dots,x_0,\dots,x_{k-2},x_{k-1}\}$ with more than $\varepsilon n$ common neighbors. By the symmetry of this $C_{2k-1}$, we can apply suitable relabeling so that for some integer $1\le p\leq k-1$, $x_{-p}$ and $x_{p}$ have more than $\varepsilon n$ common neighbors. This completes the preprocessing step.

As this $C_{2\ell-1}$ in $H$ is non-singular, there is some index $q\in\{-\ell+1,-\ell+2,\ldots,-1,0,1,\ldots,\ell-1\}$ such that $P_q^{(1)}$ is not singular. Since $P_{\ell-1}P_{-\ell+1}\in E(H)$, there exist $y_{-\ell+1}$ in class $P_{-\ell+1}$ and $z_{\ell-1}$ in class $P_{\ell-1}$ such that $y_{-\ell+1}z_{\ell-1}\in E(G)$. 
Using~\cref{claim:from hom path to real walk} for $t=\ell-1$, we obtain two walks in $G$, namely $y_0y_{-1}\cdots y_{-\ell+1}$ and $z_{\ell-1}\cdots z_1z_0$, where $y_{-i}$ belongs to class $P_{-i}^{(i+1)}$ and $z_{i}$ belongs to class $P_{i}^{(i+1)}$ for $0\le i\le\ell-1$. Since $y_{-\ell+1}z_{\ell-1}\in E(G)$, we see that $W=y_0y_{-1}\cdots y_{-\ell+1}z_{\ell-1}\cdots z_1z_0$ is a walk of length $2\ell-1$ in $G$. We distinguish two cases depending on whether $W$ is a closed walk.

Suppose that $y_0=z_0$ and $W$ is a closed walk of length $2\ell-1$ in $G$. If $W$ is a cycle in $G$, then recall that the class $P_q^{(1)}$ is not singular and so are $y_q^{(1)}=P_q^{(1)}$ if $q\le 0$ and $z_q^{(1)}=P_q^{(1)}$ if $q>0$. Therefore, the cycle $W$ is a non-singular $C_{2\ell-1}$ in $G$, a contradiction to~\cref{claim:GNonSingularCycles}. 

We may then assume that $W$ is not a cycle in $G$, in which case we must have $\ell\ge 3$. By~\cref{claim:odd walk to odd cycle}, there exists an odd cycle $C$ of length at most $2\ell-3$, where $E(C)\subseteq E(W)$ and there exists a vertex $v$ in $C$ which appears at least twice in the sequence $y_0y_{-1}\cdots y_{-\ell+1}z_{\ell-1}\cdots z_1$. Recall that $y_{-i}$ belongs to class $P_{-i}^{(i+1)}$ and $z_{i}$ belongs to class $P_{i}^{(i+1)}$ for $0\le i\le\ell-1$. Thus, there exist two indices $\alpha,\beta\in\{-\ell+1,-\ell+2,\ldots,-1,0,1,\ldots,\ell-1\}$ such that $v\in P_\alpha^{(|\alpha|+1)}\cap P_\beta^{(|\beta|+1)}$ and therefore $P_\alpha\cup P_{\beta}\subseteq v^{(1)}$. Notice that $P_\alpha$ and $P_\beta$ are two distinct classes of $\mathscr{P}_{k}(G,\frac{\varepsilon}{3})$ and hence are disjoint, which implies that 
\begin{equation}\label{eq:nonsingular}
  |v^{(1)}|\geq |P_\alpha\cup P_\beta|=|P_\alpha|+|P_\beta|\geq 2    
\end{equation}
and $v^{(1)}$ is not singular. Then $C$ is a non-singular odd cycle of length at most $2\ell-3$ in $G$, a contradiction to~\cref{claim:GNonSingularCycles}. This completes the case when $y_0=z_0$.

Suppose now that $y_0\neq z_0$. Since $y_0$ and $z_0$ belong to $P_0^{(1)}$, as in~\eqref{eq:sameclass} we have that \[|N(y_0)\cap N(z_0)|\geq \frac{2\varepsilon n}{3}\geq  |\mathscr{P}_1|+3k.\] Thus there exists a vertex $x\in \big(N(y_0)\cap N(z_0)\big)\setminus\{y_0,y_{-1},\dots ,y_{-\ell+1},z_{\ell-1},\dots ,z_1,z_0\}$ such that $x^{(1)}$ is not singular. Consider now $W'=xy_0y_{-1}\dots y_{-\ell+1}z_{\ell-1}\dots z_1z_0x$, which is a closed walk of length $2\ell+1$ in $G$.

If $W'$ is not a cycle in $G$, then by~\cref{claim:odd walk to odd cycle}, there exists a cycle $C$ of length at most $2\ell-1$ in $G$, where $E(C)\subseteq E(W)$ and there exists some vertex $v\neq x$ in $C$ that appears at least twice in $xy_0y_{-1}\cdots y_{-\ell+1}z_{\ell-1}\cdots z_1z_0$. Moreover, since $y_0\neq z_0$ and $v$ appears at least twice in $y_0y_{-1}\cdots y_{-\ell+1}z_{\ell-1}\cdots z_1z_0$, as in~\eqref{eq:nonsingular}, $v^{(1)}$ is not singular. Therefore, $C$ is a non-singular odd cycle of length at most $2\ell-1$ in $G$, a contradiction to~\cref{claim:GNonSingularCycles}.

If $W'$ is a cycle in $G$, then recall that  $x^{(1)}$ is not singular and so $W$ is a non-singular $C_{2\ell+1}$ in $G$. When $\ell\le k-1$, this contradicts~\cref{claim:GNonSingularCycles}. When $\ell=k$, this induces a copy of $C_{2k+1}$ in $G$, contradicting the second condition in~\cref{lemma:main}. 
Thus, we have arrived at the remaining case that $G$ satisfies the first condition of the lemma, i.e.,~$G$ is $\{C_{2k-1},C_{2k-3}\}$-free and $\delta(G)\ge(\frac{1}{2k-1}+ 2k\varepsilon )n$, $\ell=k$, $y_0\neq z_0$ and $W'$ is a cycle in $G$. Recall that for this case, our preprocessing step produces vertices $x_{-p}\in P_{-p}$ and $x_{p}\in P_{p}$, for some $p\in[k-1]$, which have more than $\varepsilon n$ common neighbors. As $W'=xy_0y_{-1}\dots y_{-k+1}z_{k-1}\dots z_1z_0$ is a non-singular odd cycle of length $2k+1$, $y_{-p}\cdots y_{-     k+1}z_{k-1}\cdots z_p$ is a path of length $2k-2p-1$. Since $y_{-p},x_{-p}\in P_{-p}^{(0)}$ and  $z_{p},x_{p}\in P_{p}^{(0)}$, we have $|N(y_{-p})\triangle N(x_{-p})|\leq \frac{\varepsilon n}{3}$ and $|N(z_{p})\triangle N(x_{p})|\leq \frac{\varepsilon n}{3}$, and then
\begin{equation*}
    |N(z_{p})\cap N(y_{-p})|\geq |N(x_{p})\cap N(x_{-p})|-|N(y_{-p})\triangle N(x_{-p})|-|N(z_{p})\triangle N(x_{p})|\geq \frac{\varepsilon n}{3}\geq 3k+|\mathscr{P}_1|.
\end{equation*}
By the pigeonhole principle there is a vertex $v\in N(z_{p})\cap N(y_{-p})$ such that $v^{(1)}$ is not singular and $v\notin \{y_0,y_{-1},\ldots ,y_{-k+1},z_{k-1},\ldots ,z_1,z_0\}$. Then $vy_{-p}\cdots y_{-k+1}z_{k-1}\cdots z_p$ forms a non-singular odd cycle of length $2k-2p+1\le 2k-1$ in $G$, a contradiction to~\cref{claim:GNonSingularCycles}. This finishes the proof of~\cref{lemma:main}.

\subsection{Lower bounds for~\cref{thm:SharpThreshold}}\label{sec:LowerBound}

Let $k\ge 2$ be an integer. The class $\mathscr{A}_{k}$ of Andr\'{a}sfai graphs consists of all graphs $G=(V,E)$, where $V$ is a finite subset of the unit circle $\mathbb{R}{/ }\mathbb{Z}$ and a pair of vertices in $G$ are adjacent if and only if the distance between them in $\mathbb{R}{/ }\mathbb{Z}$ is larger than $\frac{k-1}{2k-1}$. Let $r\in\mathbb{N}$, $N=(2k-1)(r-1)+2$ and $A_{k,r}$ be the graph from $\mathscr{A}_{k}$ having the corners of a regular $N$-gon as its vertices. 
Denote the vertex set of $A_{k,r}$ to be $V_{k,r}: = [N]$.
For each $x\in V_{k,r}$, the neighborhood of $x$ is $N(x): = \{ x+t \mod N: t\in [(k-1)(r-1)+1, k(r-1)+1] \}$.
It is easy to see that the graph $A_{k,r}$ is $r$-regular and $\mathscr{C}_{2k-1}$-free (see e.g.~\cite[Proposition~2.2]{2020COMBHomoOddCycle}) and has minimum degree at least $\frac{N-2}{2k-1}+1\ge\frac{N}{2k-1}$. 

We present some additional properties of $A_{k,r}$. 

\begin{prop}\label{prop:Akr}
    Let $k\ge 2$ and $r\ge 1$ be integers. The graph $A_{k,r}$ has the following properties.
    \begin{enumerate}        
        \item[\textup{(1)}] For any two distinct vertices $u$ and $v$ in $V_{k,r}$, $N(u)\neq N(v)$. Consequently, it is not a blowup of any smaller graph.
        \item[\textup{(2)}] $A_{k,r}$ is maximal $C_{2k-1}$-free.
        \item[\textup{(3)}] $\textup{VC}(A_{k,r})\le 2$.
    \end{enumerate} 
\end{prop}
\begin{proof}
    Part~(1) follows immediately from the definition of $A_{k,r}$.

    For~(2), recall that $A_{k,r}$ is $C_{2k-1}$-free. It remains to show the maximality, that is, for each pair of non-adjacent vertices $u$ and $v$ in $A_{k,r}$, $u$ and $v$ are connected by a path of length $2k-2$. Without  loss of generality, assume that $u:= 1$ and $v:= u + y$, for some $y\in [(k-1)(r-1)]$. We then aim to find a sequence $S = \{s_0, s_1, s_2,\ldots, s_{2k-2}\}$ of vertices in $V_{k,r}$ such that $s_0 = u$, $s_{2k-2} = v$ and $d_j = s_j - s_{j-1} \mod N$ lies in $[(k-1)(r-1)+1, k(r-1)+1]$ for each $j\in [2k-2]$. Let $d_j = a_j + (k-1)(r-1)+1 \mod N$, where $0\le a_j\le r-1$ for each $j\in [2k-2]$.
    Then to show the existence of such a sequence $S$, it suffices to prove the following equation has a solution.
    \begin{equation*}
        d_1 + d_2 +\cdots + d_{2k-2} = y  \mod N,
    \end{equation*}
    which is equivalent to show that
    \begin{align}\label{eq: 2k-2 path}
        a_1 + a_2 +\cdots + a_{2k-2} 
        &= y  - (2k-2)((k-1)(r-1)+1) \notag \\ 
        &= y + (k-1)(r-1)  \mod N
    \end{align}
has a non-trivial solution. On one hand, we can see that
    \begin{equation*}
        \{a_1 + a_2 +\cdots + a_{2k-2}:  0\le a_j\le r-1 \text{ for each } j\in[2k-2]\} \supseteq [(2k-2)(r-1)].
    \end{equation*}
    on the other hand,
    \begin{equation*}
        y + (k-1)(r-1) \in [(k-1)(r-1), (2k-2)(r-1)]\subseteq [(2k-2)(r-1)].
    \end{equation*}
   Therefore, for any $y\in [(k-1)(r-1)]$,  \eqref{eq: 2k-2 path} has a solution, which corresponds to a path of length $2k-2$ between $u$ and $v$ as desired.
    
    For~(3), let $S\subseteq V_{k,r}$ be the maximum shattered set in $A_{k,r}$. Suppose that $|S|\ge 3$. Take three vertices from $S$, denoted by $x_{1}<x_{2}<x_{3}$. Since $S$ is shattered, for any distinct $i,j\in [3]$, $x_{i}$ and $x_{j}$ have a common neighbor,
    which implies that 
    \begin{align*}
        |x_i - x_j|\le ((k(r-1)+1) - ((k-1)(r-1)+1)) = r-1.
    \end{align*}
  Therefore, without loss of generality, we can assume that $S\subseteq [(k-1)(r-1)+1, k(r-1)+1]$.
    Since $S$ is shattered, there exists a vertex $y\in V_{k,r}$ such that $\{x_1, x_3\}\subseteq N(y)$ while $x_2\notin N(y)$.
    By symmetry, we can assume without loss of generality that $0\le y\le (k-1)(r-1)< x_1$.
    Then we can get
    \begin{align*}
        x_1 - y \ge (k-1)(r-1)+1, \ \ 
        x_3 - y \le k(r-1)+1,
    \end{align*}
    which implies
    \begin{align*}
        (k-1)(r-1)+1 \le x_2 - y \le k(r-1)+1.
    \end{align*}
    Therefore, $x_2\in N(y)$, a contradiction.
    This completes the proof.
\end{proof}

The lower bound $\delta_{\textup{B}}(C_{2k-1})\ge\frac{1}{2k-1}$ in~\cref{thm:SharpThreshold} now immediately follows from~\cref{prop:Akr}.

\begin{rmk}\label{rmk:Kr}
Using $A_{k,r}$, we can construct, for any $r\ge 3$, an $n$-vertex $K_{r}$-free graph $G$ with VC-dimension $O(1)$ and $\delta(G)\ge \frac{2r-5}{2r-3}\cdot n$ such that $G$ has no $K_{r}$-free homomorphic image of size smaller than $\frac{3n}{2r-3}$. To see this, let $X := A_{2,\ell}$ for sufficiently large $\ell$, and let $Y_{1}, Y_{2}, \ldots, Y_{r-3}$ be $r-3$ sets, each of size $|Y_{i}| = \frac{2|X|}{3}$. Define the vertex set of $G$ as \( V(G) = X \sqcup Y_{1} \sqcup \cdots \sqcup Y_{r-3} \), and construct $G$ such that every pair of sets forms a complete bipartite graph. It is straightforward to verify that the graph $G$ is $K_{r}$-free and cannot be homomorphic to a $K_{r}$-free graph of size smaller than $|X| = \frac{3|V(G)|}{2r-3}$ by~\cref{prop:Akr}.
\end{rmk}

\section{Concluding remarks}\label{sec:ConcludingRmks}
In the first part of our paper, we introduce generalized homomorphism thresholds and apply the theory of VC-dimension to resolve this problem for cliques. Our result provides a smooth interpolation between chromatic and homomorphism thresholds for cliques. Note that, by definition, the more restrictions imposed on the homomorphic image, the larger the corresponding thresholds should be. Specifically, for a given \(\mathcal{G}_{1}\), if \(\mathcal{G}_{2} \subseteq \mathcal{G}_{2}'\), then it follows that
\[
\delta_{\textup{hom}}^{\textup{VC}}(\mathcal{G}_{1};\mathcal{G}_{2}) \leq \delta_{\textup{hom}}^{\textup{VC}}(\mathcal{G}_{1};\mathcal{G}_{2}').
\]
It would be interesting to find more examples of \(\mathcal{G}_{2} \subseteq \mathcal{G}_{2}'\) for which the above inequality is strict. For example, this is the case for $\C G_{1}=\{K_s\}$ and $\C G_2=\C K_t\subseteq \C K_s=\C G_2'$, where
$\C K_r=\{K_{\ell} \}_{\ell \geq r}$, by~\cref{thm:KsKt}; and for  $\C G_{1}=\C G_2'=\C C_{2k+1}$ and $\C G_2=\C C_{2k-1}$ by~\cref{thm:ByproductOne}. 

In the second part, we study the blowup thresholds for graphs and prove that $\delta_{\textup{B}}(C_{2k-1})=\frac{1}{2k-1}$, showing that 0 is an accumulation point for blowup thresholds. 
Recall that $\delta_{\chi}(H)\le \delta_{\textup{hom}}(H)\le \delta_{\textup{B}}(H)$. We propose the following bold yet intuitive conjecture.

\begin{conj}
    For any graph $H$, $\delta_{\textup{hom}}(H)=\delta_{\textup{B}}(H)$.
\end{conj}

The methods and results presented here motivate several other interesting directions to explore.

\subsection{Hitting all small odd cycles}
Several problems in combinatorics investigate how to destroy a specific structure by removing as few vertices as possible. For example, the famous Erd\H{o}s–P\'{o}sa theorem~\cite{1965ErdosPosa} asserts that if a graph does not contain $k$ vertex-disjoint cycles, then one can remove at most $O(k\log{k})$ vertices so that the remaining graph contains no cycles. An old conjecture of Bollob\'{a}s, Erd\H{o}s and Tuza~\cite{1991ErdosCollection} states that for any dense graph $G$ with $\alpha(G)=\Omega(|G|)$, one can remove at most $o(|G|)$ vertices to ensure the remaining graph contains no independent set of size $\alpha(G)$.

As another application of~\cref{lemma:main}, we establish a result of this type, showing that in a dense maximal $C_{2k-1}$-free graph, one can remove only a constant number of vertices to eliminate all odd cycles of length at most $2k-1$.

\begin{theorem}\label{thm:ByproductThree}
    For given integer $k\geq 3$ and any $\varepsilon>0$, let $G$ be an $n$-vertex maximal $C_{2k-1}$-free graph with minimum degree at least $(\frac{1}{2k-1}+\varepsilon)n$, then one can remove at most $O_{\varepsilon,k}(1)$ vertices to make it $\mathscr{C}_{2k-1}$-free.
\end{theorem}
\begin{proof}
    Without loss of generality, we can assume $G$ is maximal $C_{2k-1}$-free. Since $\delta(G)\ge(\frac{1}{2k-1}+\varepsilon)n$, $\textup{VC}(G)$ is at most some $d=O_{\varepsilon,k}(1)$ by~\cref{lemma:Bounded VC-dimension}. Note also that $G$ is $C_{2k-3}$-free by~\cref{lemma:NoSmallOddCycles}. Therefore $G$ satisfies the first condition in~\cref{lemma:main}, then $G$ is homomorphic to $H_{k}(G,\frac{\varepsilon}{6k})$ and moreover $G$ is also non-singular $C_{2\ell-1}$-free for $2\le\ell\le k-1$. As every singular cycles in $G$ contains vertices from singular classes of $\mathscr{P}_{1}$ and there are at most $|\mathscr{P}_{1}|=O_{\varepsilon,k}(1)$ such vertices, removing them would make $G$ $\mathscr{C}_{2k-1}$-free.
\end{proof}

It was shown in~\cite{2023Skokan3colorable} that given $k,t\in\mathbb{N}$ with $k\ge 5490+45t$, any $\mathscr{C}_{2k-1}$-free graph $G$ with minimum degree at least $\frac{|V(G)|}{2k+t}$ is $3$-colorable. This, together with the above theorem immediately implies the following.

\begin{cor}\label{cor:ByproductThree}
    For given integer $k\ge 5535$ and any $\varepsilon>0$, let $G$ be an $n$-vertex $C_{2k-1}$-free graph with minimum degree at least $(\frac{1}{2k-1}+\varepsilon)n$, there exists a subset $T$ with $|T|= O_{k,\varepsilon}(1)$ such that $\chi(G\setminus T)\le 3$.
\end{cor}

\subsection{When does a maximal $H$-free graph have bounded VC-dimension?}\label{sec:vc-threshold}
As discussed in the introduction, the bottleneck for chromatic threshold for cliques is the VC-dimension, not the minimum degree. More precisely, for a graph $H$, let $\delta_{\textup{VC}}(H)$ be the infimum of $\alpha$ such that any maximal $H$-free graph $G$ with $\delta(G)\ge \alpha n$ has bounded VC-dimension $O_{\alpha, H}(1)$. Observe that by definition, we have for every $\star\in\{\chi,\textup{hom},\textup{B}\}$,
\begin{equation}\label{eq:vc}
 \delta_{\star}(H)\le \max\{\delta_{\textup{VC}}(H),~ \delta_{\star}^{\textup{VC}}(H)\}.    
\end{equation}

For cliques, equality above holds. The difference is that for $\delta_{\chi}$, the bottleneck is $\delta_{\textup{VC}}$ as  
$$\delta_{\chi}(K_s)=\delta_{\textup{VC}}(K_s)=\frac{2s-5}{2s-3}>\frac{s-3}{s-2}=\delta_{\chi}^{\textup{VC}}(K_s),$$
whereas for $\delta_{\textup{hom}}$ and $\delta_{\textup{B}}$ all three in~\eqref{eq:vc} are the same.
Here $\delta_{\textup{VC}}(K_s)= \delta_{\chi}(K_s)$ follows from $\delta_{\chi}(K_s)>\delta_{\chi}^{\textup{VC}}(K_s)$. 

It would be interesting to understand $\delta_{\textup{VC}}$ better as its value together with~\eqref{eq:vc} shows the effect of VC-dimension for different thresholds. 

\begin{problem}\label{ques:CR}
 For given graph $H$, determine $\delta_{\textup{VC}}(H)$.
\end{problem}

In particular, for odd cycles, we have $\delta_{\textup{VC}}(C_{2k-1})\le \frac{1}{2k-1}$ by~\cref{lemma:Bounded VC-dimension}. It remains open whether $\delta_{\textup{VC}}(C_{2k-1})=\frac{1}{2k-1}$. 
In general, for any graph $H$, if a maximal $H$-free $G$ is a blowup of another graph of bounded size, then $G$ obviously has bounded VC-dimension, therefore $\delta_{\textup{VC}}(H)\le \delta_{\textup{B}}(H)$ always holds. We suspect that they always equal.
\begin{conj}
    For any graph $H$, $\delta_{\textup{VC}}(H)=\delta_{\textup{B}}(H)$.
\end{conj}

\section*{Acknowledgement}
This work was initiated during the summer program organized by the IBS ECOPRO group, which was supported by the IBS-R029-C4 grant. Xinqi Huang and Mingyuan Rong are deeply grateful to Prof. Hong Liu for providing such an excellent opportunity. The authors would like to express their gratitude to Prof. Mathias Schacht, Prof. Chong Shangguan, and Fankang He for their valuable discussions, with special thanks to Fankang He for informing them that the graphs in Sankar's construction~\cite{2022Maya} have bounded VC-dimension. Zixiang Xu would like to extend special thanks to Yixuan Zhang for the discussions regarding $C_{5}$. Mingyuan Rong would like to thank Prof. Jie Ma for helpful comments. Xinqi Huang would like to thank Prof. Xiande Zhang for helpful comments and thank Prof. Jin Yan and Prof. Yuefang Sun for attending his presentation on this topic.

\bibliographystyle{abbrv}
\bibliography{homoC2k-1}

\begin{thebibliography}{10}

\bibitem{2013advAllChromatic}
P.~Allen, J.~B\"{o}ttcher, S.~Griffiths, Y.~Kohayakawa, and R.~Morris.
\newblock The chromatic thresholds of graphs.
\newblock {\em Adv. Math.}, 235:261--295, 2013.

\bibitem{2005JCTASemi}
N.~Alon, J.~Pach, R.~Pinchasi, R.~Radoi{\v{c}}i{\'{c}}, and M.~Sharir.
\newblock Crossing patterns of semi-algebraic sets.
\newblock {\em J. Combin. Theory Ser. A}, 111(2):310--326, 2005.

\bibitem{1974ErdosSos}
B.~Andr\'{a}sfai, P.~Erd\H{o}s, and V.~T. S\'{o}s.
\newblock On the connection between chromatic number, maximal clique and minimal degree of a graph.
\newblock {\em Discrete Math.}, 8:205--218, 1974.

\bibitem{2003Bollobas}
P.~Balister, B.~Bollob\'as, O.~Riordan, and R.~H. Schelp.
\newblock Graphs with large maximum degree containing no odd cycles of a given length.
\newblock {\em J. Combin. Theory Ser. B}, 87(2):366--373, 2003.

\bibitem{2024BVCSunflower}
J.~Balogh, A.~Bernshteyn, M.~Delcourt, A.~Ferber, and H.~T. Pham.
\newblock Sunflowers in set systems with small {VC}-dimension.
\newblock {\em arXiv preprint}, arXiv: 2408.04165, 2024.

\bibitem{2024SIDMAJanzer}
C.~Beke and O.~Janzer.
\newblock On the generalized {Tur{\'a}n} problem for odd cycles.
\newblock {\em SIAM J. Discrete Math.}, 38(3):2416--2428, 2024.

\bibitem{1965BollobasSetpair}
B.~Bollob\'as.
\newblock On generalized graphs.
\newblock {\em Acta Math. Acad. Sci. Hungar.}, 16:447--452, 1965.

\bibitem{2023Skokan3colorable}
J.~B\"{o}ttcher, N.~Frankl, D.~M. Cecchelli, O.~Parczyk, and J.~Skokan.
\newblock Graphs with large minimum degree and no small odd cycles are $3$-colourable.
\newblock {\em arXiv preprint}, arXiv: 2302.01875, 2023.

\bibitem{2024BlowupVCdensity}
D.~Brada{\v{c}}, H.~Liu, Z.~Wu, and Z.~Xu.
\newblock Clique density vs blowups.
\newblock {\em arXiv preprint}, arXiv: 2410.07098, 2024.

\bibitem{1998JGT}
S.~Brandt, R.~Faudree, and W.~Goddard.
\newblock Weakly pancyclic graphs.
\newblock {\em J. Graph Theory}, 27(3):141--176, 1998.

\bibitem{2011Unpubilished}
S.~Brandt and S.~Thomass{\'e}.
\newblock Dense triangle-free graphs are four-colorable: A solution to the {E}rd{\H{o}}s-{S}imonovits problem.
\newblock preprint, 2011.

\bibitem{2025UniformBVC}
T.-W. Chao, Z.~Xu, C.~H. Yip, and S.~Zhang.
\newblock Uniform set systems with small {VC}-dimension.
\newblock {\em arXiv preprint}, arXiv: 2501.13850, 2025.

\bibitem{2024Hou}
Z.~Chen, J.~Hou, C.~Hu, and X.~Liu.
\newblock Generalized {A}ndr{\'{a}}sfai--{E}rd{\H{o}}s--{S}{\'{o}}s theorems for odd cycles.
\newblock {\em arXiv preprint}, arXiv: 2409.11950, 2024.

\bibitem{2020COMBHomoOddCycle}
O.~Ebsen and M.~Schacht.
\newblock Homomorphism thresholds for odd cycles.
\newblock {\em Combinatorica}, 40(1):39--62, 2020.

\bibitem{1959ErdosRandom}
P.~Erd\H{o}s.
\newblock Graph theory and probability.
\newblock {\em Canadian J. Math.}, 11:34--38, 1959.

\bibitem{1991ErdosCollection}
P.~Erd\H{o}s.
\newblock Problems and results on set systems and hypergraphs.
\newblock In {\em Extremal problems for finite sets ({V}isegr\'{a}d, 1991)}, volume~3 of {\em Bolyai Soc. Math. Stud.}, pages 217--227. J\'{a}nos Bolyai Math. Soc., Budapest, 1994.

\bibitem{1965ErdosPosa}
P.~Erd\H{o}s and L.~P\'osa.
\newblock On independent circuits contained in a graph.
\newblock {\em Canadian J. Math.}, 17:347--352, 1965.

\bibitem{1966ES}
P.~Erd\H{o}s and M.~Simonovits.
\newblock A limit theorem in graph theory.
\newblock {\em Studia Sci. Math. Hungar.}, 1:51--57, 1966.

\bibitem{1973ErdosSimonovits}
P.~Erd\H{o}s and M.~Simonovits.
\newblock On a valence problem in extremal graph theory.
\newblock {\em Discrete Math.}, 5:323--334, 1973.

\bibitem{1946ErodsBAMS}
P.~Erd\H{o}s and A.~H. Stone.
\newblock On the structure of linear graphs.
\newblock {\em Bull. Amer. Math. Soc.}, 52:1087--1091, 1946.

\bibitem{1935ES}
P.~Erd\H{o}s and G.~Szekeres.
\newblock A combinatorial problem in geometry.
\newblock {\em Compositio Math.}, 2:463--470, 1935.

\bibitem{2012Crelle}
J.~Fox, M.~Gromov, V.~Lafforgue, A.~Naor, and J.~Pach.
\newblock Overlap properties of geometric expanders.
\newblock {\em J. Reine Angew. Math.}, 671:49--83, 2012.

\bibitem{2019DCGBVCEH}
J.~Fox, J.~Pach, and A.~Suk.
\newblock Erd{\H{o}}s-{H}ajnal conjecture for graphs with bounded {VC}-dimension.
\newblock {\em Discrete Comput. Geom.}, 61(4):809--829, 2019.

\bibitem{2021ErdosSchur}
J.~Fox, J.~Pach, and A.~Suk.
\newblock Bounded {VC}-dimension implies the {S}chur-{E}rd{\H{o}}s conjecture.
\newblock {\em Combinatorica}, 41(6):803--813, 2021.

\bibitem{2023BVCSunflower}
J.~Fox, J.~Pach, and A.~Suk.
\newblock Sunflowers in set systems of bounded dimension.
\newblock {\em Combinatorica}, 43(1):187--202, 2023.

\bibitem{2024FranklPach}
G.~Ge, Z.~Xu, C.~H. Yip, S.~Zhang, and X.~Zhao.
\newblock The {F}rankl-{P}ach upper bound is not tight for any uniformity.
\newblock {\em arXiv preprint}, arXiv: 2412.11901, 2024.

\bibitem{2011JGTKrChromatic}
W.~Goddard and J.~Lyle.
\newblock Dense graphs with small clique number.
\newblock {\em J. Graph Theory}, 66(4):319--331, 2011.

\bibitem{1982Haggkvist}
R.~H\"aggkvist.
\newblock Odd cycles of specified length in nonbipartite graphs.
\newblock In {\em Graph theory ({C}ambridge, 1981)}, volume~62 of {\em North-Holland Math. Stud.}, pages 89--99. North-Holland, Amsterdam-New York, 1982.

\bibitem{1995PackingLemma}
D.~Haussler.
\newblock Sphere packing numbers for subsets of the {B}oolean {$n$}-cube with bounded {V}apnik-{C}hervonenkis dimension.
\newblock {\em J. Combin. Theory Ser. A}, 69(2):217--232, 1995.

\bibitem{1995DMJin}
G.~P. Jin.
\newblock Triangle-free four-chromatic graphs.
\newblock {\em Discrete Math.}, 145(1-3):151--170, 1995.

\bibitem{2019JGTC3C5}
S.~Letzter and R.~Snyder.
\newblock The homomorphism threshold of {$\{C_3,C_5\}$}-free graphs.
\newblock {\em J. Graph Theory}, 90(1):83--106, 2019.

\bibitem{2024Wang}
H.~Lin, G.~Wang, and W.~Zhou.
\newblock A strengthening on consecutive odd cycles in graphs of given minimum degree.
\newblock {\em arXiv preprint}, arXiv: 2410.00648, 2024.

\bibitem{2024GraphToGeom}
H.~Liu, C.~Shangguan, J.~Skokan, and Z.~Xu.
\newblock Beyond the chromatic threshold via {$(p,q)$}-theorem, and a sharp blow-up phenomenon.
\newblock {\em arXiv preprint}, arXiv: 2403.17910, 2024.

\bibitem{2025StabChromaticThresholds}
H.~Liu, C.~Shangguan, and Y.~Xue.
\newblock Stability for chromatic thresholds.
\newblock In preparation.

\bibitem{2006CombTriangleHom}
T.~{\L}uczak.
\newblock On the structure of triangle-free graphs of large minimum degree.
\newblock {\em Combinatorica}, 26(4):489--493, 2006.

\bibitem{2010ColoringViaVCDim}
T.~{\L}uczak and S.~Thomass{\'e}.
\newblock Coloring dense graphs via {V}{C}-dimension.
\newblock arXiv preprint: 1007.1670, 2010.

\bibitem{2007JACMZ}
D.~Mubayi and Y.~Zhao.
\newblock On the {VC}-dimension of uniform hypergraphs.
\newblock {\em J. Algebraic Combin.}, 25(1):101--110, 2007.

\bibitem{2023BVCErdosHajnal}
T.~Nguyen, A.~Scott, and P.~Seymour.
\newblock Induced subgraph density. {V}{I}. {B}ounded {VC}-dimension.
\newblock {\em arXiv preprint}, arXiv: 2312.15572, 2023.

\bibitem{2010arxivKrfree}
V.~Nikiforov.
\newblock Chromatic number and minimum degree of ${K}_{r}$-free graphs, 2010.
\newblock arXiv preprint: 1001.2070.

\bibitem{2004JGTNikiforov}
V.~Nikiforov and R.~H. Schelp.
\newblock Cycles and paths in graphs with large minimal degree.
\newblock {\em J. Graph Theory}, 47(1):39--52, 2004.

\bibitem{2020CPCProb}
H.~Oberkampf and M.~Schacht.
\newblock On the structure of dense graphs with bounded clique number.
\newblock {\em Comb. Probab. Comput.}, 29(5):641--649, 2020.

\bibitem{2022Maya}
M.~Sankar.
\newblock Homotopy and the homomorphism threshold of odd cycles.
\newblock arXiv preprint 2206.07525, 2022.

\bibitem{SchachtPC}
M.~Schacht.
\newblock Personal communication.

\bibitem{2023SukMatchingLemma}
A.~Suk.
\newblock On short edges in complete topological graphs.
\newblock {\em arXiv preprint}, arXiv: 2307.08165, 2023.

\bibitem{1978OriginalRegularity}
E.~Szemer\'{e}di.
\newblock Regular partitions of graphs.
\newblock In {\em Probl\`emes combinatoires et th\'{e}orie des graphes ({C}olloq. {I}nternat. {CNRS}, {U}niv. {O}rsay, {O}rsay, 1976)}, volume 260 of {\em Colloq. Internat. CNRS}, pages 399--401. CNRS, Paris, 1978.

\bibitem{2002Thomassen}
C.~Thomassen.
\newblock On the chromatic number of triangle-free graphs of large minimum degree.
\newblock {\em Combinatorica}, 22(4):591--596, 2002.

\bibitem{2007OddCycleChromatic}
C.~Thomassen.
\newblock On the chromatic number of pentagon-free graphs of large minimum degree.
\newblock {\em Combinatorica}, 27(2):241--243, 2007.

\bibitem{2024PengC2k+1}
Z.~Yan, Y.~Peng, and X.~Yuan.
\newblock Tight minimum degree condition to guarantee {$C_{2k+1}$}-free graphs to be {$r$}-partite.
\newblock {\em arXiv preprint}, arXiv: 2409.03407, 2024.

\bibitem{2024JGTPeng}
X.~Yuan and Y.~Peng.
\newblock Minimum degree stability of {$C_{2k+1}$}-free graphs.
\newblock {\em J. Graph Theory}, 106(2):307--321, 2024.

\end{thebibliography}

\end{document}